\documentclass[reqno]{amsart}

\usepackage{eqnarray, amsfonts, amsmath, amssymb, hyperref, url, amsthm, tikz-cd, stmaryrd, graphicx, booktabs, verbatim, mathrsfs, subfigure, enumerate}

\numberwithin{equation}{section}

\theoremstyle{norm}
\newtheorem{thm}{Theorem}[section]
\newtheorem{lem}[thm]{Lemma}

\newtheorem{prop}[thm]{Proposition}
\newtheorem{df}[thm]{Definition}
\newtheorem{conj}[thm]{Conjecture}
\newtheorem{cor}[thm]{Corollary}
\newtheorem{rem}[thm]{Remark}

\newcommand{\Hm}{H\underline{\mathbb{F}_2}}
\newcommand{\HZ}{H\underline{\mathbb{Z}}}
\newcommand{\F}{\mathbb{F}_2}
\newcommand{\as}{a_\sigma}
\newcommand{\us}{u_\sigma}
\newcommand{\EMaySS}{C_2\text{-MaySS}}
\newcommand{\EASS}{C_2\text{-ASS}}
\newcommand{\Sq}{\text{Sq}}

\DeclareMathOperator{\Ext}{\text{Ext}}

\DeclareMathOperator{\ASS}{\text{ASS}}
\DeclareMathOperator{\SliceSS}{\text{SliceSS}}
\DeclareMathOperator{\HFPSS}{\text{HFPSS}}

\title[Hurewicz Images of Real Bordism Theory and $E\mathbb{R}(n)$]{Hurewicz Images of Real Bordism Theory and Real Johnson--Wilson Theories}

\author{Guchuan Li}
\address{Department of Mathematics, Northwestern University, Evanston, IL 60208}
\email{guchuanli2013@u.northwestern.edu}

\author{XiaoLin Danny Shi}
\address{Department of Mathematics, Harvard University, Cambridge, MA 02138}
\email{dannyshi@math.harvard.edu}

\author{Guozhen Wang}
\address{Shanghai Center for Mathematical Sciences, Fudan University, Shanghai, China, 200433}
\email{wangguozhen@fudan.edu.cn}

\author{Zhouli Xu}
\address{Department of Mathematics, Massachusetts Institute of Technology, Cambridge, MA, 02142}
\email{xuzhouli@mit.edu}

\begin{document}

\maketitle

\begin{abstract}
We show that the Hopf elements, the Kervaire classes, and the $\bar{\kappa}$-family in the stable homotopy groups of spheres are detected by the Hurewicz map from the sphere spectrum to the $C_2$-fixed points of the Real bordism spectrum.  A subset of these families is detected by the $C_2$-fixed points of Real Johnson--Wilson theory $E\mathbb{R}(n)$, depending on $n$.  In the proof, we establish an isomorphism between the slice spectral sequence and the $C_2$-equivariant May spectral sequence of $BP_\mathbb{R}$.  
\end{abstract}

\tableofcontents


\section{Introduction}

\subsection{Motivation and main results}

In 2009, Hill, Hopkins, and Ravenel resolved a longstanding open problem in algebraic topology.  In their seminal paper \cite{HHRKervaire}, they showed that the Kervaire invariant elements $\theta_j$ do not exist for $j \geq 7$ (see also \cite{HaynesKervaire, HHRCDM1, HHRCDM2} for surveys on the result).  The crux of their proof relies on a detecting spectrum $\Omega$, which detects the Kervaire invariant elements.

\begin{thm}[Hill--Hopkins--Ravenel Detection Theorem]
If $\theta_j \in \pi_{2^{j+1}-2} \mathbb{S}$ is an element of Kervaire invariant 1, and $j >2$, then the Hurewicz image of $\theta_j$ under the map $\pi_*\mathbb{S} \to \pi_*\Omega$ is nonzero.  
\end{thm}

The detecting spectrum $\Omega$ is constructed as the $C_8$-fixed point of a genuine $C_8$-equivariant spectrum $\Omega_\mathbb{O}$, which is an equivariant localization of $MU^{((C_8))} := N_{C_2}^{C_8}MU_\mathbb{R}$.  Here, $MU_\mathbb{R}$ is the Real cobordism spectrum of Landweber, Fujii, and Araki \cite{LandweberMUR, FujiiMUR, Araki} and $N_{C_2}^{C_8}(-)$ is the Hill--Hopkins--Ravenel norm functor.  To analyze the equivariant homotopy groups of $\Omega_\mathbb{O}$, Hill, Hopkins, and Ravenel generalized the $C_2$-equivariant filtration of Hu--Kriz (\cite{HuKriz}) and Dugger (\cite{DuggerKR}) to a $G$-equivariant Postnikov filtration for all finite groups $G$.  They called this the slice filtration.  Given any $G$-equivariant spectrum $X$, the slice filtration produces the slice tower $\{P^*X\}$, whose associated slice spectral sequence is strongly convergent and converges to the $RO(G)$-graded homotopy groups $\pi_\bigstar^G X$.  Using the slice spectral sequence, Hill, Hopkins, and Ravenel proved that
$$\displaystyle \pi_{2^{j+1}-2} \Omega = \pi_{2^{j+1}-2}^{C_8} \Omega_\mathbb{O} = 0$$
for all $j \geq 7$, hence deducing the nonexistence of the corresponding Kervaire invariant elements.\\

We are interested in proving more detection theorems for the fixed points of the equivariant theories $MU^{((C_{2^n}))} := N_{C_2}^{C_{2^n}}MU_\mathbb{R}$ and their localizations.  Our motivation is as follows: classically, $\pi_*MU$ is a polynomial ring, hence torsion free, and the map $\pi_*\mathbb{S} \to \pi_*MU$ detects no nontrivial elements in the stable homotopy groups of spheres.  Equivariantly, however, computations of Hu--Kriz \cite{HuKriz}, Dugger \cite{DuggerKR}, Kitchloo--Wilson \cite{KitchlooWilsonFibration}, and Hill--Hopkins--Ravenel \cite{HHRKervaire, HHRKH} show that there are many torsion classes in the equivariant homotopy groups of the theories above.  Since the Kervaire invariant elements are  detected by the fixed point of a localization of $MU^{((C_8))}$, there should be other classes in the stable homotopy groups of spheres that are also detected by such theories.  We prove this is indeed the case.

\begin{thm}[Theorem~\ref{thm:MURBPRDetection}, Detection Theorems for $MU_\mathbb{R}$ and $BP_\mathbb{R}$]\label{INTROthm:MURBPRDetection}
The Hopf elements, the Kervaire classes, and the $\bar{\kappa}$-family (see Definition~\ref{introDef:kappabar}) are detected by the Hurewicz maps $\pi_*\mathbb{S} \to \pi_* MU_\mathbb{R}^{C_2}$ and $\pi_*\mathbb{S} \to \pi_* BP_\mathbb{R}^{C_2}$.  
\end{thm}

Once we obtain the detection theorem for $\pi_*MU_\mathbb{R}^{C_2}$, we use the Hill--Hopkins--Ravenel norm functor to show that these elements are also detected by the $C_{2^n}$-fixed point of ${MU^{((C_{2^n}))}}$: 

\begin{cor}[Corollary~\ref{cor:DetectionMUG}, Detection Theorem for $MU^{((G))}$]\label{INTROcor:DetectionMUG} 
For any finite group $G$ containing $C_2$, the $G$-fixed point of $MU^{((G))}$ detects the Hopf elements, the Kervaire classes, and the $\bar{\kappa}$-family. 
\end{cor}

We pause here to discuss some implications of Theorem~\ref{INTROthm:MURBPRDetection}, as well as what we mean by the ``$\bar{\kappa}$-family''.  It is well known that the Hopf elements are represented by the elements 
$$h_i \in \text{Ext}_{\mathcal{A}_*}^{1, 2^i}(\mathbb{F}_2, \mathbb{F}_2)$$ 
on the $E_2$-page of the classical Adams spectral sequence at the prime 2.  By Adams's solution of the Hopf invariant one problem \cite{AdamsHopfInvariant}, only $h_0$, $h_1$, $h_2$, and $h_3$ survive to the $E_\infty$-page.  By Browder's work \cite{Browder}, the Kervaire classes $\theta_j \in \pi_{2^{j+1}-2}\mathbb{S}$, if they exist, are represented by the elements 
$$h_j^2 \in \text{Ext}_{\mathcal{A}_*}^{2, 2^{j+1}}(\mathbb{F}_2, \mathbb{F}_2)$$
on the $E_2$-page.  For $j \leq 5$, $h_j^2$ survives.  The case $\theta_4 \in \pi_{30} \mathbb{S}$ is due to Barratt, Mahowald, and Tangora \cite{MahowaldTangoraTheta4, BarratMahowaldTangoraTheta4}, and the case $\theta_5 \in \pi_{62} \mathbb{S}$ is due to Barratt, Jones, and Mahowald \cite{BarrattJonesMahowaldTheta5}. The fate of $h_6^2$ is unknown.  Hill, Hopkins, and Ravenel \cite{HHRKervaire} showed that the $h_j^2$, for $j \geq 7$, cannot survive to the $E_\infty$-page.  Given this information, Theorem~\ref{INTROthm:MURBPRDetection} and Corollary~\ref{INTROcor:DetectionMUG} assert that the elements $\eta$, $\nu$, $\sigma$, and $\theta_j$, for $1 \leq j \leq 5$, are detected by $\pi_*^G MU^{((G))}$.  The last unknown Kervaire class, $\theta_6$, will also be detected, should it survive the Adams spectral sequence.  

To introduce the $\bar{\kappa}$-family, we appeal to Lin's complete classification of the groups $\Ext^{\leq 4, t}_{\mathcal{A}_*}(\F, \F)$ \cite{LinExt}.  In his classification, Lin showed that there is a family $\{g_k \,| \, k \geq 1\}$ of indecomposable elements with 
$$g_k \in \Ext^{4, 2^{k+2} + 2^{k+3}}_{\mathcal{A}_*}(\F, \F).$$
The first element of this family, $g_1$, is in bidegree $(4, 24)$.  It survives the Adams spectral sequence to become $\bar{\kappa} \in \pi_{20} \mathbb{S}$.  It is for this reason that we name this family the $\bar{\kappa}$-family.  The element $g_2$ also survives to become the element $\bar{\kappa}_2 \in \pi_{44}\mathbb{S}$.  Theorem~\ref{INTROthm:MURBPRDetection} and Corollary~\ref{INTROcor:DetectionMUG} assert that they are both detected by $\pi_*^G MU^{((G))}$.  Recent computations of Isaksen--Wang--Xu \cite{IsaksenWangXu} show that $g_3$ supports a nontrivial $d_5$-differential and therefore $\bar{\kappa}_3$ does not exist in $\pi_{92}\mathbb{S}$.  For $k \geq 4$, the fate of $g_k$ is unknown ($g_4$ is in stem 188).  Nevertheless, they will be detected by $\pi_*^G MU^{((G))}$, should they survive the Adams spectral sequence. 

\begin{df}\rm \label{introDef:kappabar}
The \textit{$\bar{\kappa}$-family} consists of the homotopy classes detected by the surviving $g_k$-family.  
\end{df}

To prove Theorem~\ref{INTROthm:MURBPRDetection}, first observe that 2-locally, $MU_\mathbb{R}$ splits as a wedge of suspensions of $BP_\mathbb{R}$.  Therefore we only need to prove the claim for $BP_\mathbb{R}$.  To establish the link between the famlies $\{h_i\}$, $\{h_j^2\}$, and $\{g_k\}$ and the equivariant homotopy groups of $BP_\mathbb{R}$, we use the $C_2$-equivariant Adams spectral sequence developed by Greenlees \cite{GreenleesThesis, Greenlees1988, Greenlees1990} and Hu--Kriz \cite{HuKriz}.  More precisely, we analyze the following maps of Adams spectral sequences 
$$\begin{tikzcd} 
\text{classical Adams spectral sequence of } \mathbb{S} \ar[r, Rightarrow] \ar[d] & (\pi_*\mathbb{S})^\wedge_2 \ar[d] \\ 
C_2\text{-equivariant Adams spectral sequence of }\mathbb{S} \ar[r, Rightarrow] \ar[d] & (\pi_\bigstar^{C_2} \mathbb{S})^\wedge_2 \ar[d] \\ 
C_2\text{-equivariant Adams spectral sequence of } {BP_\mathbb{R}} \ar[r, Rightarrow] & (\pi_\bigstar^{C_2} BP_\mathbb{R})^\wedge_2
\end{tikzcd}$$
and prove the following. 
\begin{thm}[Algebraic Detection Theorem]\label{INTROthm:algebraicDetection}
The images of the elements $\{h_i \,|\, i \geq 1\}$, $\{h_j^2 \,|\, j \geq 1\}$, and $\{g_k \,|\, k \geq 1\}$ on the $E_2$-page of the classical Adams spectral sequence of $\mathbb{S}$ are nonzero on the $E_2$-page of the $C_2$-equivariant Adams spectral sequence of $BP_\mathbb{R}$.  
\end{thm}
It turns out that for degree reasons, the $C_2$-equivariant Adams spectral sequence of $BP_\mathbb{R}$ degenerates after the $E_2$-page.  From this, Theorem~\ref{INTROthm:MURBPRDetection} easily follows from Theorem~\ref{INTROthm:algebraicDetection} because if any of $h_i$, $h_j^2$, or $g_k$ survives to the $E_\infty$-page of the classical Adams spectral sequence to represent an element in the stable homotopy groups of spheres, it must be detected by $\pi_\bigstar^{C_2}BP_\mathbb{R}$.  \\

The proof of Theorem~\ref{INTROthm:algebraicDetection} requires us to analyze the algebraic maps 
$$\text{Ext}_{\mathcal{A}_*}(\mathbb{F}_2, \mathbb{F}_2) \to \text{Ext}_{\mathcal{A}_\bigstar^m}(\Hm_\bigstar, \Hm_\bigstar) \to \text{Ext}_{\Lambda_\bigstar^m}(\Hm_\bigstar, \Hm_\bigstar).$$
They are maps on the $E_2$-pages of the Adams spectral sequences above.  Here, $\mathcal{A}_*:= (H\mathbb{F}_2 \wedge H\mathbb{F}_2)_*$ is the classical dual Steenrod algebra; $\mathcal{A}_\bigstar^m:= (\Hm \wedge \Hm)_\bigstar$ is the genuine $C_2$-equivariant dual Steenrod algebra; $\Lambda_\bigstar^m$ is a quotient of $\mathcal{A}_\bigstar^m$.  Hu and Kriz \cite{HuKriz} studied $\mathcal{A}_\bigstar^m$ and completely computed the Hopf algebroid structure of $(\Hm_\bigstar, \mathcal{A}_\bigstar^m)$.  We borrow extensively their formulas.  More precisely, we use their formulas to describe the maps
$$(H\mathbb{F}_2, \mathcal{A}_*) \to (\Hm_\bigstar, \mathcal{A}_\bigstar^m) \to (\Hm_\bigstar, \Lambda_\bigstar^m)$$
 of Hopf-algebroids.  Then, by filtering these Hopf algebroids compatibly, we produce maps of May spectral sequences: 
$$\begin{tikzcd}
\text{Modified May spectral sequence of } \mathbb{S} \ar[r, Rightarrow] \ar[d] & \text{Ext}_{\mathcal{A}_*}(\mathbb{F}_2, \mathbb{F}_2) \ar[d] \\ 
C_2\text{-equivariant May spectral sequence of }\mathbb{S} \ar[r, Rightarrow] \ar[d] & \text{Ext}_{\mathcal{A}_\bigstar^m}(\Hm_\bigstar, \Hm_\bigstar) \ar[d] \\ 
C_2\text{-equivariant May spectral sequence of } {BP_\mathbb{R}} \ar[r, Rightarrow] &  \text{Ext}_{\Lambda_\bigstar^m}(\Hm_\bigstar, \Hm_\bigstar).
\end{tikzcd}$$
To analyze these maps, we appeal again to Hu and Kriz's formulas.  We compute the maps on the $E_2$-page of the May spectral sequences above, as well as all the differentials in the $C_2$-equivariant May spectral sequence of $BP_\mathbb{R}$.  

The readers should be warned that the May spectral sequence at the top of the diagram above is \textit{not} the classical May spectral sequence.  The classical May spectral sequence is constructed from an increasing filtration of the dual Steenrod algebra $\mathcal{A}_*$.  However, in constructing the equivariant May spectral sequence, we filtered $\mathcal{A}^m_\bigstar$ and $\Lambda^m_\bigstar$ by decreasing filtrations.  To rectify this mismatch of filtrations, we need to change the filtration of $\mathcal{A}_*$ to a decreasing filtration as well.  This is necessary to ensure the compatibility of filtrations with respect to the map $\mathcal{A}_* \to \mathcal{A}^m_\bigstar \to \Lambda^m_\bigstar$ --- or we won't have a map of spectral sequences.  Nevertheless, despite this change of filtration, we are able to compute this modified May spectral sequence.  This computation, together with our knowledge of the $C_2$-equivariant May spectral sequence of $BP_\mathbb{R}$, finishes the proof of Theorem~\ref{INTROthm:algebraicDetection}. \\

While proving Theorem~\ref{INTROthm:algebraicDetection}, we also prove a connection between the equivariant May spectral sequence of $BP_\mathbb{R}$ and the slice spectral sequence of $BP_\mathbb{R}$.

\begin{thm}[Theorem~\ref{thm:AmSliceSS}] \label{INTROthm:AmSliceSS}
The integer-graded $C_2$-equivariant May spectral sequence of $BP_\mathbb{R}$ is isomorphic to the associated-graded slice spectral sequence of $BP_\mathbb{R}$.  
\end{thm}

By the ``associated-graded slice spectral sequence'', we mean that whenever we see a $\mathbb{Z}$-class on the $E_2$-page, we replace it by a tower of $\mathbb{Z}/2$-classes.  Theorem~\ref{INTROthm:AmSliceSS} can be intuitively explained as follows: since the Adams spectral sequence for $BP_\mathbb{R}$ collapses for degree reasons, the equivariant May spectral sequence of $BP_\mathbb{R}$ converges to an associated-graded of $(\pi_\bigstar BP_\mathbb{R})^{\wedge}_2$.  On the other hand, the slice spectral sequence 
$$E_2^{s, V} = \pi_{V-s}^{C_2} P^{\dim V}_{\dim V} BP_\mathbb{R} \Longrightarrow \pi_\bigstar^{C_2} BP_\mathbb{R} $$
also computes the equivariant homotopy groups of $BP_\mathbb{R}$.  Moreover, works of \cite{HuKriz} and \cite{HHRCDM2} essentially show that the $C_2$-slice differentials are produced from equivariant cohomology operations.  Given this, one should naturally suspect the isomorphism in Theorem~\ref{INTROthm:AmSliceSS}.  As we will discuss shortly, Theorem~\ref{INTROthm:AmSliceSS} is crucial in tackling detection theorems for Real Johnson--Wilson theories.  \\

The homotopy groups of the fixed point spectra $(MU^{((C_{2^n}))})^{C_{2^n}}$ can be assembled into the commutative diagram 
$$\begin{tikzcd}
&& \vdots \ar[d]& \\
&&\pi_*(MU^{((C_{2^n}))})^{C_{2^n}} \ar[d] \\
&& \vdots \ar[d] \\ 
\pi_*\mathbb{S} \ar[rr] \ar[rrd] \ar[rrdd] \ar[rruu]&&\pi_*(MU^{((C_8))})^{C_{8}} \ar[d] \\ 
&&\pi_*(MU^{((C_4))})^{C_{4}} \ar[d] \\ 
&&\pi_*(MU_\mathbb{R})^{C_{2}}.
\end{tikzcd}$$

As we move up the tower, more and more elements in the stable homotopy groups of spheres are detected by $\pi_*(MU^{((C_{2^n}))})^{C_{2^n}}$.  For instance, in \cite{HHRKH}, Hill, Hopkins, and Ravenel completely computed the Mackey functor homotopy groups of $K_{[2]}$, the $C_4$-analogue of Atiyah's $K\mathbb{R}$-theory \cite{AtiyahKR}.  The spectrum $K_{[2]}$ is the periodization (localization) of a quotient of $MU^{((C_4))}$, and the $C_4$-action on the underlying spectrum of $K_{[2]}$ is compatible with the $C_4$-action on $E_2$ (the height two Morava $E$-theory spectrum), where $C_4 \subset G_{24} \subset \mathbb{S}_2$.  Here, $\mathbb{S}_2$ is the second Morava stabilizer group, and $G_{24}$ is the maximal finite subgroup, which is of order 24.  Using this, Hill, Hopkins, and Ravenel deduced that $\eta \in \pi_1 \mathbb{S}$, $\nu \in \pi_3\mathbb{S}$, $\epsilon \in \pi_8 \mathbb{S}$, $\kappa \in \pi_{14} \mathbb{S}$, and $\bar{\kappa} \in \pi_{20}\mathbb{S}$ are detected by $\pi_*^{C_4} K_{[2]}$.  Of these elements, $\epsilon$ and $\kappa$ are not detected by $\pi_*^{C_2} MU_\mathbb{R}$.  It is a current project to generalize the techniques developed in this paper to prove detection theorems for the $G$-fixed points of $MU^{((G))}$ for $|G| > 2$. \\

The Doomsday Conjecture claims that for any $s$, there are only finitely many surviving permanent cycles in $\text{Ext}^{s, t}_{\mathcal{A}_*}(\mathbb{F}_2, \mathbb{F}_2)$.  This was proven false by Mahowald in 1977.  In particular, Maholwald exhibited a family of infinitely many surviving permanent cycles on the 2-line of the classical Adams spectral sequence.  In 1995, Minami modified the Doomsday conjecture.
\begin{conj}[New Doomsday Conjecture]
For any $Sq^0$-family $$\{x, Sq^0(x), \ldots, (Sq^0)^n(x), \ldots \}$$ in $\Ext_{A_*}(\mathbb{F}_2, \mathbb{F}_2)$, only finitely many elements survive to the $E_\infty$-page of the classical Adams spectral sequence. 
\end{conj}
Here, $Sq^0(-)$ is the Steenrod action defined on the Adams $E_2$-page (see \cite{BrunerMayHinfty}).  In particular, the families $\{h_i \,|\, i \geq1\}$, $\{h_j^2 \,|\, j \geq 1 \}$, and $\{g_k \,|\, k \geq 1 \}$ are all $Sq^0$-families on the 1-line, 2-line, and 4-line of the classical Adams spectral sequence, respectively.  We are interested in the fate of the $\bar{\kappa}$-family $\{g_k \, |\, k \geq 1 \}$ in $\pi^G_* MU^{((G))}$ as we increase the order of $G$.  As $G$ grows bigger, it's possible that $g_k$ will all support differentials in the slice spectral sequence of $\pi^G_* MU^{((G))}$ for $k$ large enough, hence not surviving the classical Adams spectral sequence.\\  

In \cite{HHRKervaire}, Hill, Hopkins, and Ravenel also used an algebraic detection theorem to prove that the Kervaire classes are detected by $\pi_*^{C_8}\Omega_\mathbb{O}$.  They remarked that their algebraic detection theorem can be modified to prove that the $C_{2^n}$-fixed points of $MU^{((C_{2^n}))}$, for $n \geq 3$, detect the Kervaire classes.  It's worth pointing out the differences between our algebraic detection theorem and their algebraic detection theorem.  To prove their detection theorem, Hill, Hopkins, and Ravenel used the map of spectral sequences 
$$\begin{tikzcd} 
\text{Adams--Novikov spectral sequence} \ar[r, Rightarrow]  \ar[d] & \pi_* \mathbb{S} \ar[d] \\ 
C_{2^n}\text{-homotopy fixed point spectral sequence} \ar[r, Rightarrow] & \pi_* (MU^{((C_{2^n}))})^{C_{2^n}}.
\end{tikzcd}$$
Their algebraic detection theorem \cite[Theorem~11.2]{HHRKervaire} shows that if $x \in \text{Ext}_{MU_*MU}^{2, 2^{j+1}} (MU_*, MU_*)$ is any element mapping to $h_j^2$ on the $E_2$-page of the classical Adams spectral sequence, then the image of $x$ in $H^2(C_{2^n}, \pi_{2^{j+1}}^u MU^{((C_{2^n}))})$ is not zero.  Once this is proved, their detection theorem follows easily.  

They further remarked that their algebraic detection theorem does not hold when $G$ is $C_2$ or $C_4$ (see \cite[Remark~11.14]{HHRKervaire}).  For these groups, there is a jump of filtration.  In particular, for $n = 1$, the element $x \in \text{Ext}_{MU_*MU}^{2, 2^{j+1}}(MU_*, MU_*)$ maps to 0 on the $E_2$-page of the $C_2$-homotopy fixed point spectral sequence of $MU_\mathbb{R}$.  However, because of Theorem~\ref{INTROthm:AmSliceSS}, we deduce that there must be a nontrivial extension so that $x$ actually corresponds to an element of filtration $2^{j+1}-2$ in the $C_2$-homotopy fixed point spectral sequence.  For our algebraic detection theorem, this jump of filtration does not occur because we used maps of Adams spectral sequences. \\

As an application of Theorem~\ref{INTROthm:MURBPRDetection}, we study Hurewicz images of Real Johnson--Wilson theories.  The Real Johnson--Wilson theories $E\mathbb{R}(n)$ were first constructed and studied by Hu and Kriz \cite{HuKriz}.  They constructed $E\mathbb{R}(n)$ from $BP_\mathbb{R}$ by mimicking the classical construction of $E(n)$.  More precisely, there is an isomorphism
$$\mathbb{Z}[v_1, v_2, \ldots] = \pi_{2*} BP \cong \pi_{*\rho_2}^{C_2} BP_\mathbb{R} = \mathbb{Z} [\bar{v}_1, \bar{v}_2, \ldots],$$
where $\bar{v}_i \in \pi_{i\rho_2}^{C_2} BP_\mathbb{R}$ are lifts of the classical generators $v_i \in \pi_{2i} BP$.  Quotienting out the $\bar{v}_i$ generators for all $i \geq n + 1$ and inverting $\bar{v}_n$ produces the Real Johnson--Wilson theory $E\mathbb{R}(n)$.  It is a $C_2$-equivariant spectrum whose underlying spectrum is $E(n)$, with an $C_2$-action induced from the complex conjugation action of $BP_\mathbb{R}$. 

Many people have also studied $E\mathbb{R}(n)$ after Hu and Kriz.  Kitchloo and Wilson \cite{KitchlooWilsonFibration} proved that the fixed points $ER(n) := E\mathbb{R}(n)^{C_2}$ fits into the fiber sequence
$$\Sigma^{\lambda(n)} ER(n) \to ER(n) \to E(n),$$
where $\lambda(n) = 2^{2n+1} - 2^{n+2} + 1$.  When $n = 1$, $E\mathbb{R}(1)$ is Atiyah's Real $K$-theory $K\mathbb{R}$, with $ER(1) = K\mathbb{R}^{C_2} = KO$.  In this case, Kitchloo and Wilson's fibration recovers the classical fibration 
$$\Sigma KO \stackrel{\eta}{\to} KO \to KU.$$
When $n=2$, using the Bockstein spectral sequence associated to the fibration, Kitchloo and Wilson subsequently computed the cohomology groups $ER(2)^*(\mathbb{R}P^n)$ and $ER(2)^*(\mathbb{R}P^n \wedge \mathbb{R}P^m)$.  From their computation, they deduced new nonimmersion results for even dimensional real projective spaces \cite{KitchlooWilsonImmersion1, KitchlooWilsonImmersion2}.  Most recently, Kitchloo, Lorman, and Wilson have used this Bockstein spectral sequence to further compute the $ER(n)$ cohomology of other spaces as well \cite{Lorman, KitchlooLormanWilson1, KitchlooLormanWilson2}. 

In \cite{HillMeier}, Hill and Meier studied the spectra $TMF_1(3)$ and $Tmf_1(3)$ of topological modular forms at level three.  They proved that the spectrum $tmf_1(3)$, considered as an $C_2$-equivariant spectrum, is a form of $BP_\mathbb{R}\langle 2\rangle $, and $tmf_1(3)[\bar{a}_3^{-1}]$ is a form of $E\mathbb{R}(2)$.  Using this identification, they computed the $C_2$-equivariant Picard groups and the $C_2$-equivariant Anderson dual of $Tmf_1(3)$.  

We are interested in the Hurewicz images of $\pi_* E\mathbb{R}(n)^{C_2}$.  To do so, we study the map of slice spectral sequences
$$\SliceSS(BP_\mathbb{R}) \to \SliceSS(E\mathbb{R}(n)).$$
Theorem~\ref{INTROthm:algebraicDetection} and Theorem~\ref{INTROthm:AmSliceSS} identify the classes in the slice spectral sequence of $BP_\mathbb{R}$ that detect the families $\{h_i\}$, $\{h_j^2\}$, and $\{g_k\}$.  Analyzing the images of these classes in the slice spectral sequence of $E\mathbb{R}(n)$ produces the detection theorem for $E\mathbb{R}(n)$.

\begin{thm}[Detection Theorem for $E\mathbb{R}(n)$] \hfill\label{INTROthm:ER(n)Detection}
\begin{enumerate}
\item For $1 \leq i, j \leq n$, if the element $h_{i} \in \Ext_{\mathcal{A}_*}^{1, 2^i}(\F, \F)$ or $h_{j}^2 \in \Ext_{\mathcal{A}_*}^{2, 2^{j+1}}(\F, \F)$ survives to the $E_\infty$-page of the Adams spectral sequence, then its image under the Hurewicz map $\pi_*\mathbb{S} \to \pi_* E\mathbb{R}(n)^{C_2}$ is nonzero.  
\item For $1 \leq k \leq n-1$, if the element $g_k \in \Ext_{\mathcal{A}_*}^{4, 2^{k+2}+2^{k+3}}(\F, \F)$ survives to the $E_\infty$-page of the Adams spectral sequence, then its image under the Hurewicz map $\pi_*\mathbb{S} \to \pi_* E\mathbb{R}(n)^{C_2}$ is nonzero. 
\end{enumerate}
\end{thm}

Theorem~\ref{INTROthm:ER(n)Detection} is extremely useful for computing  $ER(n)^*(\mathbb{RP}^m)$.  In \cite{LiShiWangXu}, we use the fact that the Hopf elements are detected by $\pi_*^{C_2} E\mathbb{R}(n)$ to deduce the compatibility of the slice differentials of $E\mathbb{R}(n)$ and the attaching maps of $\mathbb{RP}^m$.  As a result, we are able to compute $ER(2)^*(\mathbb{RP}^m)$ by a double filtration spectral sequence, solving all the $2$-extensions and some $\eta$ and $\nu$-extensions.  

Hahn and the second author have shown that the Lubin--Tate theories $E_n$, equipped with the Goerss--Hopkins--Miller $C_2$-action (\cite{HopkinsMiller, GoerssHopkins}), is Real oriented.  In other words, there is a $C_2$-equivariant map $MU_\mathbb{R} \to E_n$.  The proof for Theorem~\ref{INTROthm:ER(n)Detection} can be modified to prove Hurewicz images for the homotopy fixed point spectra $E_n^{hG}$.  In \cite{HahnShi}, the authors show that the Hurewicz images of $E\mathbb{R}(n)^{C_2}$ and $E_n^{hC_2}$ are the same.  It follows that Theorem~\ref{INTROthm:ER(n)Detection} holds for $\pi_*E_n^{hC_2}$ as well. 

\subsection{Summary of the contents} In Section~\ref{dualSteenrod}, we provide the necessary background for the $C_2$-equivariant dual Steenrod algebras --- $(\Hm_\bigstar, \mathcal{A}^m_\bigstar)$ and $(H^c_\bigstar, \mathcal{A}^{cc}_\bigstar)$ --- and their $C_2$-equivariant Adams spectral sequences.  In Section~\ref{sec:SliceSSHFPSSBPR}, we compute the slice spectral sequence and the homotopy fixed point spectral sequence of $BP_\mathbb{R}$.  In Section~\ref{sec:EquivMaySS}, we construct the equivariant May spectral sequence of $BP_\mathbb{R}$ and prove Theorem~\ref{INTROthm:AmSliceSS}.  In Section~\ref{sec:MaySSMay}, we modify the filtration of the classical dual Steenrod algebra $\mathcal{A}_*$ to obtain a compatible filtration with respect to the map $\mathcal{A}_* \to \mathcal{A}^m_\bigstar \to \mathcal{A}^{cc}_\bigstar$ of Steenrod algebras.  We then analyze the resulting maps of May spectral sequences.  Lastly, in Section~\ref{sec:DetectionTheorem}, we combine results from the previous sections and prove Theorem~\ref{INTROthm:MURBPRDetection}, Corollary~\ref{INTROcor:DetectionMUG}, Theorem~\ref{INTROthm:algebraicDetection}, and Theorem~\ref{INTROthm:ER(n)Detection}.  

\subsection{Acknowledgements}
The authors would like to thank the organizers of the 2016 Talbot workshop, Eva Belmont, Inbar Klang, and Dylan Wilson, for inviting them to the workshop.  This project would not have come into being without the mentorship of Mike Hill and Doug Ravenel during the workshop.  We would like to thank Vitaly Lorman for helpful conversations and a fruitful exchange of ideas.  We are also grateful to Hood Chatham for his spectral sequence package, which produced all of our diagrams.  Thanks are also due to Mark Behrens, Jeremy Hahn, Achim Krause, Peter May, Haynes Miller, Eric Peterson, Doug Ravenel, David B Rush, and Mingcong Zeng for helpful conversations.  Finally, we would like to heartily thank Mike Hill and Mike Hopkins for sharing numerous insights with us during various stages of the project and many helpful conversations.  The fourth author was partially supported by the National Science Foundation under Grant No. DMS-1810638.

\section{The Equivariant Dual Steenrod Algebra and Adams Spectral Sequence} \label{dualSteenrod}

In this section, we provide the necessary background for the $C_2$-equivariant dual Steenrod algebra and the $C_2$-equivariant Adams spectral sequence.  These have been extensively studied by Hu--Kriz \cite{HuKriz} and Greenlees \cite{GreenleesThesis, Greenlees1988, Greenlees1990}.  Of the many ways to define the $C_2$-equivariant dual Steenrod algebra, two of them are of interest to us.  The first one is the \textbf{Borel equivariant dual Steenrod algebra} 
$$\mathcal{A}^{cc}_\bigstar := F(E{C_2}_+, H\F \wedge H\F)_\bigstar.$$
This has been studied by Greenlees.  The second one is the \textbf{genuine equivariant dual Steenrod algebra}.  It is defined by using the genuine Eilenberg--Mac Lane spectrum $\Hm$:
$$\mathcal{A}^m_\bigstar := (\Hm \wedge \Hm)_\bigstar.$$


\subsection{$\mathcal{A}_\bigstar^{cc}$ and $\mathcal{A}_\bigstar^m$}

To compute $\mathcal{A}^m_\bigstar$, Hu and Kriz first computed the $RO(C_2)$-graded homotopy groups $\mathcal{A}^{cc}_\bigstar$.  This computation can be further used to deduce the $RO(C_2)$-graded homotopy groups $\mathcal{A}^m_\bigstar$.  We give a brief summary of Hu and Kriz's computation of $\mathcal{A}_\bigstar^{cc}$ and $\mathcal{A}_\bigstar^m$, focusing on the parts that we will need again for the later sections.  For more details of their computation, see Section 6 of \cite{HuKriz}. \\


To start, we need the coefficient rings of the $C_2$-equivariant Eilenberg--Mac Lane spectra $H^c := F(E{C_2}_+, H\mathbb{F}_2)$ and $H^m := \Hm$.  The following are some distinguished elements in their $RO(C_2)$-graded homotopy groups. 

\begin{df}\rm
The element 
$$a_\sigma \in \pi_{-\sigma}^{C_2} S^0$$ 
is the element corresponding to the inclusion $S^0 \hookrightarrow S^\sigma$ (the one point compactification of the inclusion $\{0\} \subset \sigma$) under the suspension isomorphism $\pi_{-\sigma}^{C_2} S^0 \cong \pi_0^{C_2} S^{\sigma}$.  Under the Hurewicz maps $\pi^{C_2}_\bigstar \mathbb{S} \to H^c_\bigstar$ and $\pi^{C_2}_\bigstar \mathbb{S} \to H^m_\bigstar$, the images of $a_\sigma$ are nonzero.  By an abuse of notation, we will denote the images by $a_\sigma$ as well.  
\end{df}

\begin{df}\rm
The element 
$$u_\sigma \in \pi_{1 - \sigma}^{C_2} \Hm$$ 
is the element corresponding to the generator of $H_1^{C_2} (S^\sigma ;\underline{\mathbb{F}_2})=\pi_1^{C_2} (S^\sigma \wedge \Hm)$.  It can also be regarded as an element in $\pi_{1 -\sigma}^{C_2} F(E{C_2}_+, H\mathbb{F}_2)$ via the map 
$$\Hm \to F(E{C_2}_+, \Hm) \simeq F(E{C_2}_+, H\mathbb{F}_2).$$
\end{df}

Hu and Kriz first computed $H^c_\bigstar$.  They then used it to analyze the cofiber of the map $$\Hm \to F(E{C_2}_+, \Hm) \simeq F(E{C_2}_+, H\mathbb{F}_2)$$ and subsequently computed the coefficient ring $\Hm_\bigstar$. 

\begin{prop}[Hu--Kriz] \hfill
\begin{enumerate}
\item The coefficient ring $H^c_\bigstar = F(E{C_2}_+, H\mathbb{F}_2)_\bigstar$ is the polynomial algebra 
$$F(E{C_2}_+, H\mathbb{F}_2)_{p + q\sigma} = \mathbb{F}_2[u_\sigma^\pm, a_\sigma].$$

\item The coefficient ring $H^m_\bigstar := \Hm_\bigstar$ is 
$$\Hm_{p + q\sigma} = \mathbb{F}_2 [u_\sigma, a_\sigma] \oplus \mathbb{F}_2\left\{ \frac{\theta}{u_\sigma^i a_\sigma^j} \right\}, \, \, \, \, \, i, j \geq 0,$$
where $\theta$ is an element in $\pi_{2\sigma-2}^{C_2} \Hm$.  The element $\theta$ is infinitely $u_\sigma$ and $a_\sigma$-divisible.  It is also $u_\sigma$ and $a_\sigma$-torsion.  The product of any two elements $x, y \in \mathbb{F}_2\left\{ \frac{\theta}{u_\sigma^i a_\sigma^j} \right\}$ is 0.  
\end{enumerate}
In particular, the map $H^m_{p+q\sigma} \to H^c_{p+q\sigma}$ is an isomorphism in the range $a \geq 0$.  Figure~\ref{fig:HcCoeffHmCoeff} shows $\Hm_{p+q\sigma}$ and $F(E{C_2}_+, H\mathbb{F}_2)_{p+q\sigma}$.
\end{prop} 


\begin{figure}
\begin{center}
\includegraphics[trim={1cm 15cm 11cm 4cm},clip, scale = 0.65]{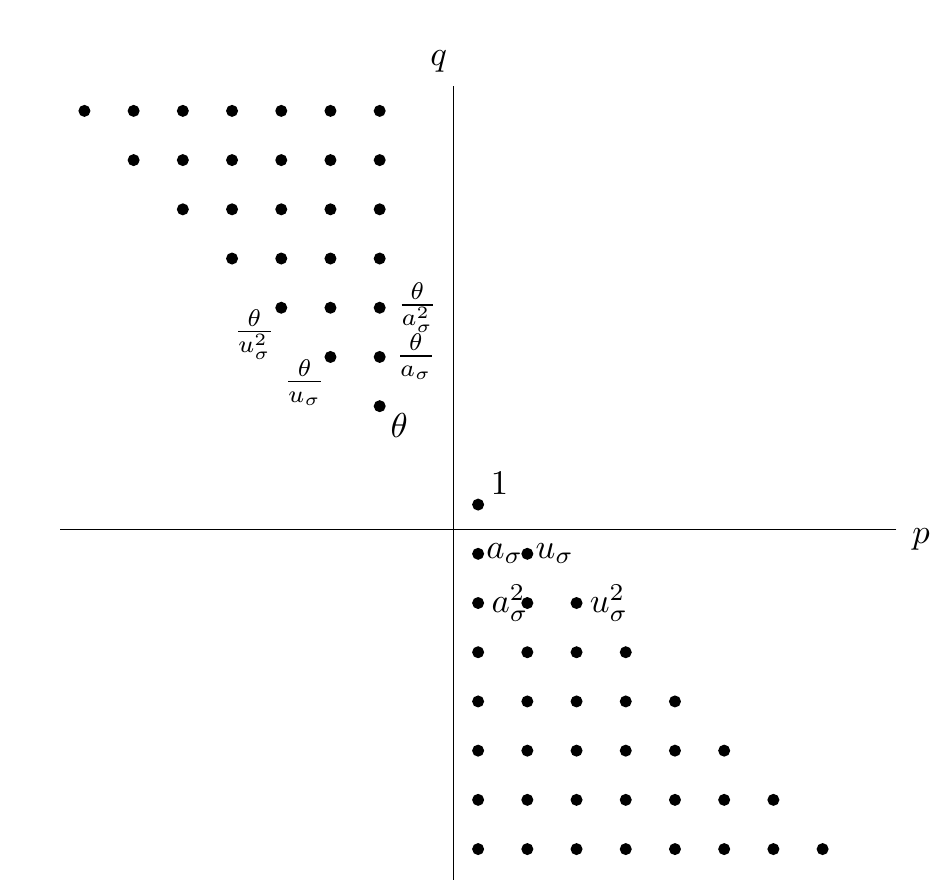}
\includegraphics[trim={1cm 15cm 11cm 4cm},clip, scale = 0.65]{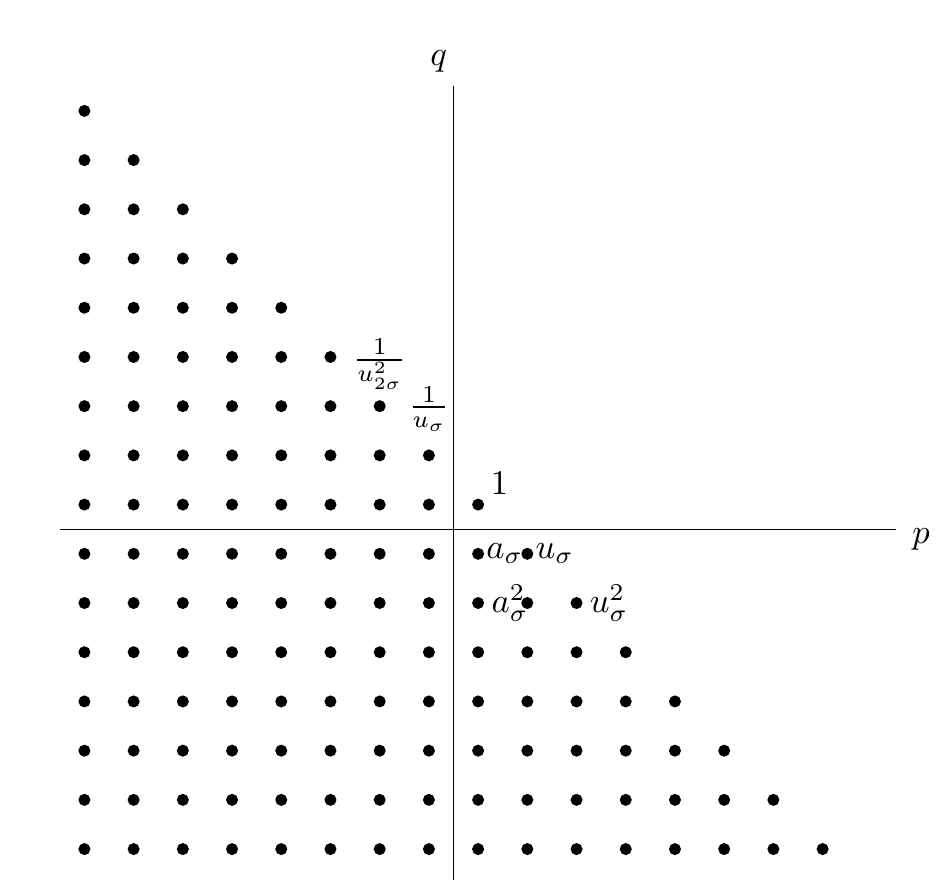}
\end{center}
\caption{The coefficient rings of $\Hm_{p+q\sigma}$ (left) and $F(E{C_2}_+, H\mathbb{F}_2)_{p+q\sigma}$ (right).  The map $\Hm_{p+q\sigma} \longrightarrow F(E{C_2}_+, H\mathbb{F}_2)_{p+q\sigma}$ induces an isomorphism in the range $p \geq 0$.}
\label{fig:HcCoeffHmCoeff}
\end{figure}

\begin{rem}\rm
In \cite{HuKriz}, Hu and Kriz denoted $u_\sigma$ by $\sigma^{-1}$ and $a_\sigma$ by $a$.  
\end{rem}

\begin{rem}\rm \label{rem:theta}
The element $\theta \in H_\bigstar^m$ can be defined as follows: consider the Tate diagram 
$$\begin{tikzcd}
E{C_2}_+ \wedge \Hm \ar[r] \ar[d, "\simeq"] &\Hm  \ar[r] \ar[d] & \widetilde{E{C_2}} \wedge \Hm \ar[d] \\ 
E{C_2}_+ \wedge F(E{C_2}_+, \Hm) \ar[r] & F(E{C_2}_+, \Hm) \ar[r] & \widetilde{EC_2}_+ \wedge F(E{C_2}_+, \Hm).
\end{tikzcd}$$
Taking $\pi_\bigstar^{C_2}(-)$ produces the following diagram on homotopy groups
$$\begin{tikzcd}
H_\bigstar^f \ar[r] \ar[d, "\cong"] & H_\bigstar^m \ar[r] \ar[d] & H^g_\bigstar \ar[d] \label{diagram:Tate}\\ 
H_\bigstar^f \ar[r] & H_\bigstar^c \ar[r] & H^t_\bigstar.
\end{tikzcd}$$
The coefficient rings of $H^t_\bigstar$ and $H^f_\bigstar$ can be computed to be
\begin{eqnarray}
H_\bigstar^t &=& \mathbb{F}_2[u_\sigma^\pm, \as^\pm] \nonumber \\ 
H_\bigstar^f &=& \mathbb{F}_2[u_\sigma^\pm, \as^{-1}]. \nonumber 
\end{eqnarray}
Now, consider the boundary map $\partial$ defined by using the long exact sequences of homotopy groups for the top and bottom rows of the Tate diagram: 
$$\partial: H_\bigstar^t \to H_{\bigstar -1}^f \to H_{\bigstar-1} ^m.$$
The element $\theta$ is the image of $\us^{-1}\as^{-1} \in H^t_\bigstar$ under the boundary map $\partial$.  
\end{rem}


With these coefficient groups in hand, we are now ready to compute the equivariant dual Steenrod algebras.  When computing $\mathcal{A}_\bigstar^{cc}$, we need to work in the category $\mathscr{M}$ of bigraded $\mathbb{Z}[a_\sigma]$-modules that are complete with respect to the topology associated with the principal ideal $(a_\sigma)$.  The morphisms in this category are continuous homomorphisms.  It turns out that even though $\mathcal{A}_\bigstar^{cc}$ is not flat over $H_\bigstar^c$ as $\mathbb{F}_2$-modules, completion by $(a_\sigma)$ ensures that $\mathcal{A}_\bigstar^{cc}$ \textit{is} flat over $H_\bigstar^c$ in the category $\mathscr{M}$.  Thus, we can regard $(H_\bigstar^c, \mathcal{A}_\bigstar^{cc})$ as a $\mathscr{M}$-Hopf algebroid.  

\begin{thm}[Hu--Kriz] \label{thm:HuKriz6.10}
The $\mathscr{M}$-Hopf algebroid $(H_\bigstar^{c}, \mathcal{A}_\bigstar^{cc})$ can be described by the following structure formulas: 
\begin{enumerate}
\item $\displaystyle \mathcal{A}_\bigstar^{cc} = H_\bigstar^c [\zeta_i \, | \, i \geq 1]^{\wedge}_{a_\sigma}$, $\dim \zeta_i = 2^i -1$;
\item $\displaystyle \psi(\zeta_i) = \sum_{0 \leq j \leq i} \zeta_{i-j}^{2^j} \otimes \zeta_j$, with $\zeta_0 = 1$;
\item $\displaystyle \eta_R(u_\sigma^{-1}) = \sum_{i \geq 0} (u_\sigma^{-1})^{2^i} \zeta_i a_\sigma^{2^i -1}$;
\item $\displaystyle \eta_R(a_\sigma) = a_\sigma$.
\end{enumerate}
\end{thm}

The formula for $\mathcal{A}_\bigstar^{cc}$ in Theorem~\ref{thm:HuKriz6.10} is obtained from the $RO(C_2)$-graded homotopy fixed point spectral sequence 
$$H^s(C_2; \pi_t(H\mathbb{F}_2 \wedge H\mathbb{F}_2) \otimes \text{sgn}^{\otimes r}) \Longrightarrow \pi_{t-s + (\sigma -1)r} F(E{C_2}_+, H\mathbb{F}_2 \wedge H\mathbb{F}_2),$$
where $\text{sgn}$ is the sign representation.  The $\zeta_i$ generators in $\mathcal{A}_\bigstar^{cc}$ are images of the $\zeta_i$ generators in the classical dual Steenrod algebra $\mathcal{A}_* = \mathbb{F}_2[\zeta_i \, | \, i \geq 1]$ under the map 
$$H\mathbb{F}_2 \wedge H \mathbb{F}_2 \to F(E{C_2}_+, H\mathbb{F}_2 \wedge H \mathbb{F}_2).$$
Although Theorem~\ref{thm:HuKriz6.10} provides formulas describing $(H_\bigstar^c, \mathcal{A}_\bigstar^{cc})$ as a $\mathscr{M}$-Hopf algebroid, it is not very helpful for computing the Hopf algebroid structure of $(H_\bigstar^m, \mathcal{A}_\bigstar^m)$.  To further compute $\mathcal{A}_\bigstar^m$, Hu and Kriz constructed explicit equivariant generators $\xi_i$ and $\tau_i$ in both $\mathcal{A}^m_\bigstar$ and $\mathcal{A}^{cc}_\bigstar$. These generators are compatible in the sense that under the map $\mathcal{A}^m_\bigstar \to \mathcal{A}^{cc}_\bigstar$, $\tau_i \mapsto \tau_i$ and $\xi_i \mapsto \xi_i$.  By computing the relations between the $\xi_i$'s and $\tau_i$'s, Hu and Kriz obtained an alternative description of the $\mathscr{M}$-Hopf algebroid structure of $(H_\bigstar^c, \mathcal{A}_\bigstar^{cc})$.  Afterwards, they observed that the exact same relations hold in $\mathcal{A}_\bigstar^m$ as well.  This observation ultimately led them to conclude the Hopf algebroid structure of $(\Hm_\bigstar, \mathcal{A}^m_\bigstar)$.  \\

We now introduce the $\xi_i$ and $\tau_i$ generators.  We structure our exposition to focus on describing the map 
$$\mathcal{A}_* \to \mathcal{A}^m_\bigstar.$$
Understanding this map will be of great importance to us later on.

\begin{df}\rm
For $X$ an $C_2$-equivariant spectrum, let 
\begin{eqnarray*}
H_\bigstar^{cc} X &:=& F(E{C_2}_+, H\mathbb{F}_2 \wedge X)_\bigstar, \\ 
H_\bigstar^m X &:=& (\Hm \wedge X)_\bigstar. 
\end{eqnarray*}
\end{df}

Classically, if a spectrum $E$ is complex oriented, then one can easily compute $E_*BP$ as follows: choose a complex orientation $b \in E^2 (\mathbb{CP}^\infty)$.  Associated to $b$ is a coproduct formula
$$\psi(b) = {\sum_{i \geq 0}}^F b^{2^i} \otimes \xi_i,$$
where $F$ is formal sum induced by the complex orientation of $E$.  From this coproduct formula, one is led to conclude that 
$$E_*BP = E_*[\xi_i \, | \, i \geq 1], \, \, \, |\xi_i| = 2(2^i -1).$$
This argument works $C_2$-equivariantly as well.  The genuine Eilenberg--Mac Lane spectrum $\Hm$ is Real oriented via the Thom map $BP_\mathbb{R} \to \Hm$.  Applying the argument above produces equivariant polynomial generators for $H_\bigstar^{cc} BP_\mathbb{R}$ and $H_\bigstar^m BP_\mathbb{R}$.  

\begin{prop} [Hu--Kriz] \label{prop:BPRxi}
There exist generators $\xi_i$ of dimensions $|\xi_i| = (2^i -1) \rho_2$ in both $H_\bigstar^{m} BP_\mathbb{R}$ and $H_\bigstar^{cc} BP_\mathbb{R}$, such that 
\begin{eqnarray}
H_\bigstar^m BP_\mathbb{R} &=& H_\bigstar^m [\xi_i \, | \, i \geq 1], \nonumber \\
H_\bigstar^{cc} BP_\mathbb{R} &=& H_\bigstar^c [\xi_i \, | \, i \geq 1]. \nonumber 
\end{eqnarray}
Furthermore, the two sets of $\xi_i$ generators are compatible in the sense that the map 
$$\Hm \wedge BP_\mathbb{R} \to F(E{C_2}_+, \Hm \wedge BP_\mathbb{R}) = F(E{C_2}_+, H\mathbb{F}_2 \wedge BP_\mathbb{R}) $$
induces the map 
$$H_\bigstar^m BP_\mathbb{R} \to H_\bigstar^{cc} BP_\mathbb{R}$$
sending $\xi_i \mapsto \xi_i$. 
\end{prop}

\begin{df}\rm
The orientation map $f: BP_\mathbb{R} \to \Hm$ induces the commutative diagram 
$$\begin{tikzcd}
\Hm \wedge BP_\mathbb{R} \ar[d] \ar[r, "\text{id} \wedge f"] & \Hm \wedge \Hm \ar[d] \\ 
F(E{C_2}_+, \Hm \wedge BP_\mathbb{R}) \ar[r] & F(E{C_2}_+, \Hm \wedge \Hm),
\end{tikzcd}$$
which, after taking equivariant homotopy groups $\pi_\bigstar^{C_2}(-)$, becomes
$$\begin{tikzcd} 
H^m_\bigstar BP_\mathbb{R} \ar[r] \ar[d] & \mathcal{A}^m_\bigstar \ar[d] \\ 
H^{cc}_\bigstar BP_\mathbb{R} \ar[r] & \mathcal{A}^{cc}_\bigstar.
\end{tikzcd}$$
The image of the $\xi_i$ generators in Proposition~\ref{prop:BPRxi} produces generators $\xi_i \in \mathcal{A}_\bigstar^m$ and $\mathcal{A}_\bigstar^{cc}$.  
\end{df}

Consider the commutative diagram 
$$\begin{tikzcd}
H\mathbb{F}_2 \wedge H\mathbb{F}_2 = \Hm^{C_2} \wedge \Hm^{C_2} \ar[r] \ar[rd] & (\Hm \wedge \Hm)^{C_2} \ar[d] \\ 
& F(E{C_2}_+, \Hm \wedge \Hm)^{C_2} \ar[d, equal] \\
& F(E{C_2}_+, H\mathbb{F}_2 \wedge H\mathbb{F}_2)^{C_2}.
\end{tikzcd}$$
Taking $\pi_*(-)$ produces the diagram
$$\begin{tikzcd}
\mathcal{A}_* \ar[r] \ar[rd] & \mathcal{A}_*^m \ar[d] \\
& \mathcal{A}_*^{cc}. 
\end{tikzcd}$$
Here, $\mathcal{A}_*^m$ and $\mathcal{A}_*^{cc}$ are the integer graded parts of $\mathcal{A}_\bigstar^m$ and $\mathcal{A}_\bigstar^{cc}$, respectively.  The following theorem provides formulas relating the $\xi_i$ generators and images of the $\zeta_i$ generators under the maps $\mathcal{A}_* \to \mathcal{A}^m_* \to \mathcal{A}^{cc}_*$.  

\begin{thm} [Relations between $\xi_i$ and $\zeta_i$] \hfill \label{thm:xizeta} 
\begin{enumerate}
\item $ \displaystyle \psi (\xi_i) = \sum_{0 \leq j \leq i} \xi_{i-j}^{2^j} \otimes \xi_j$;
\item The $\xi_i$ generators are related to the images of the $\zeta_i$ generators (which, by an abuse of notation, will also be denoted by $\zeta_i$) by the recursion formulas 
\begin{eqnarray}
\xi_0 &=& 1 \nonumber \\ 
a_\sigma^{2^i} \xi_i &=& \zeta_{i-1}^2 \eta_R(\us) + \zeta_i \as + \xi_{i-1} \us^{2^{i-1}}, \,\, \,  i \geq 1.
\end{eqnarray}
\end{enumerate}
\end{thm}
\begin{proof}
We prove the relations in $\mathcal{A}^m_\bigstar$.  Once we have proven that they hold in $\mathcal{A}^m_\bigstar$, they will automatically hold in $\mathcal{A}^{cc}_\bigstar$ as well.  The proof is essentially the same as the proof of Theorem 6.18 in \cite{HuKriz}.  Let $b \in \Hm^{\rho_2} (BS^1_+)$ be the Real orientation and $r \in \Hm^1(B\mathbb{Z}/2_+)$ be the generator of $\Hm^\bigstar (B\mathbb{Z}/2_+)$.  The coproduct formulas for $b$ and $r$ are, by definition,   
\begin{eqnarray}
\psi(b) &=& \sum_{i \geq 0} b^{2^i} \otimes \xi_i  \nonumber \\ 
\psi(r) &=& \sum_{i \geq 0 } r^{2^i} \otimes \zeta_i. \nonumber 
\end{eqnarray}
Part (1) is obtained by computing $\psi(\psi(b))$ in two ways through the commutative diagram 
$$\begin{tikzcd}
\Hm^\bigstar(BS^1_+) \ar[r, "\psi"] \ar[d, "\psi"] & \Hm^\bigstar(BS^1_+) \wedge \mathcal{A}_\bigstar^m \ar[d, "\text{id} \wedge \psi"] \\
\Hm^\bigstar(BS^1_+)  \wedge \mathcal{A}_\bigstar^m \ar[r, "\psi \wedge \text{id}"] & \Hm^\bigstar(BS^1_+) \wedge \mathcal{A}_\bigstar^m \wedge \mathcal{A}_\bigstar^m
\end{tikzcd}$$
and comparing the coefficients of $b^{2^i}$: 
$$\sum_{i \geq 0} \psi(b)^{2^i} \otimes \xi_i = \psi(\psi(b)) = \sum_{i \geq 0} b^{2^i} \otimes \psi(\xi_i).$$
For part (2), the map $B\mathbb{Z}/2_+ \to BS^1_+$ induces the map
$$\Hm^\bigstar(BS^1_+) \to \Hm^\bigstar(B\mathbb{Z}/2_+)$$
on equivariant cohomologies.  This is a map 
$$\Hm_\bigstar[[b]] \to \Hm_\bigstar[[r]],$$
where $|b| = -\rho_2$, $|r| = -1$.  We would like to express the image of $b$ in terms of $r$.  The only terms on the right hand side that are of degree $-\rho_2$ are $r^2u_\sigma$ and $ra_\sigma$.  Hu and Kriz show that $b$ maps to the sum of these two terms: 
$$b \mapsto r^2 u_\sigma + ra_\sigma.$$
The commutative diagram
$$\begin{tikzcd}
\Hm^\bigstar(BS_+^1) \ar[r, "f"] \ar[d, "\psi"] & \Hm^\bigstar(B\mathbb{Z}/2_+) \ar[d, "\psi"] \\ 
\Hm^\bigstar(BS_+^1) \wedge \mathcal{A}_\bigstar^m \ar[r, "f \wedge \text{id}"]& \Hm^\bigstar(B\mathbb{Z}/2_+) \wedge \mathcal{A}_\bigstar^m,
\end{tikzcd}$$
obtained by the naturality of the coproduct, implies that 
$$\sum_{i \geq 0} (r^2\us + r\as)^{2^i} \otimes \xi_i =  \psi(b) = \psi(r^2u_\sigma + ra_\sigma) =\eta_R(\us) \sum_{i \geq 0} r^{2^{i+1}} \otimes \zeta_i^2 + a_\sigma \sum_{i \geq 0} r^{2^i} \otimes \zeta_i.$$
Comparing coefficients of $r^{2^i}$ on both sides produces the recursion formulas, as desired.
\end{proof}

We will now define the $\tau_i$ generators and compute their relations to the images of the classical $\zeta_i$ generators.  Consider the $C_2$-equivariant map $BS^1 \to BS^1$ classifying the squaring of Real line bundles.  This produces the fiber sequence
\begin{eqnarray}
B_\mathbb{R} \mathbb{Z}/2 \to BS^1 \stackrel{L^2}{\to} BS^1 \label{BRZ2},
\end{eqnarray}
where the fiber $B_\mathbb{R} \mathbb{Z}/2$ is $\mathbb{RP}^\infty$, but with a nontrivial $C_2$-action (the fixed point of $B_\mathbb{R} \mathbb{Z}/2$ under the $C_2$ action is $\mathbb{RP}^\infty \coprod \mathbb{RP}^\infty$).  The Real orientation $b \in H\underline{\mathbb{Z}}^{\rho_2}(BS^1)$ restricts to a class $b' \in H\underline{\mathbb{Z}}^{\rho_2} (B_\mathbb{R} \mathbb{Z}/2)$.  Under the map $\HZ \to \Hm$, this gives a class $b'' \in \Hm^{\rho_2} (B_\mathbb{R} \mathbb{Z}/2)$:  
$$\begin{tikzcd}
b \in \HZ^{\rho_2}(\mathbb{CP}^\infty)  \ar[r, mapsto] \ar[d, mapsto] & b' \in \HZ^{\rho_2}(B_\mathbb{R} \mathbb{Z}/2) \ar[d, mapsto] \\
b \in \Hm^{\rho_2} (\mathbb{CP}^\infty) \ar[r, mapsto] & b'' \in \Hm^{\rho_2} (B_\mathbb{R} \mathbb{Z}/2).
\end{tikzcd}$$

The composition 
$$\HZ^\bigstar(BS^1) \to \HZ^\bigstar(BS^1)  \to \HZ^\bigstar(B_\mathbb{R} \mathbb{Z}/2) $$
sends $b \mapsto 2b' = 0$.  This implies that $b'$ is in the image of the Bockstein $\beta$, induced by $\HZ \stackrel{2}{\to} \HZ \to \Hm$: 
$$\cdots \to \Hm^{\sigma} (B_\mathbb{R} \mathbb{Z}/2) \stackrel{\beta}{\to} \HZ^{\rho_2} (B_\mathbb{R} \mathbb{Z}/2) \stackrel{\cdot 2}{\to} \HZ^{\rho_2}(B_\mathbb{R} \mathbb{Z}/2) \to \cdots .$$
Let $c \in \Hm^{\sigma}(B_\mathbb{R} \mathbb{Z}/2)$ be a class such that $\beta c = b'$.  

\begin{prop}[Hu--Kriz]
$\Hm^\bigstar (B_\mathbb{R} \mathbb{Z}/2)$ is a free $\Hm^\bigstar$-module with basis $\{1, b, b^2, \ldots, c, cb, cb^2, \ldots \}$.  
\end{prop}
\begin{proof}
Consider the cofiber sequence 
$$B_\mathbb{R} \mathbb{Z}/2_+ \to BS^1_+ \to \text{Thom}(BS^1, L^2). $$
Taking $\Hm^\bigstar(-)$ produces the Gysin sequence 
$$\Hm^\bigstar(\text{Thom}(BS^1, L^2)) \stackrel{0}{\to} \Hm^\bigstar(BS^1_+) \to \Hm^\bigstar(B_\mathbb{R} \mathbb{Z}/2_+).$$
By the Thom isomorphism theorem, $\Hm^\bigstar(\text{Thom}(BS^1, L^2)) \cong \Hm^\bigstar(BS^1_+)[x]$ as a free $\Hm^\bigstar(BS^1_+)$-module.  The generator $x \in \Hm^{\rho_2}(\text{Thom}(BS^1, L^2))$ maps to $0 \in \Hm^\bigstar(BS_+^1)$, and it is the image of $c \in \Hm^{\sigma}(B_\mathbb{R} \mathbb{Z}/2_+)$.  It follows that as a $\Hm_\bigstar$-module, 
$$\Hm^\bigstar(B_\mathbb{R} \mathbb{Z}/2_+) \cong \Hm^\bigstar(BS_+^1) \oplus \Hm^\bigstar(BS_+^1)[c] = \Hm_\bigstar\{1, b, b^2, \ldots, c, cb, cb^2, \ldots\}.$$

\end{proof}

Since $\psi(x) = x \otimes 1$ and $c \mapsto x$, the coproduct formula for $c$ must be of the form 
$$\psi(c) = c \otimes 1 + \sum_{i \geq 0} (b'')^{2^i} \otimes \tau_i, $$
where $\tau_i$ are elements in $\mathcal{A}_\bigstar^m$ with dimensions $|\tau_i| = (2^i -1) \rho_{C_2} +1$.

\begin{thm} [Hu--Kriz] \hfill \label{thm:tauzeta}
\begin{enumerate}
\item $\displaystyle \psi(\tau_i) = \tau_i \otimes 1 + \sum_{0 \leq j \leq i} \xi_{i-j}^{2^j} \otimes \tau_j$.
\item The $\tau_i$ generators are related to the images of the $\zeta_i$ generators by the recursion formulas 
\begin{eqnarray}
a_\sigma \tau_0 &=& \eta_R(u_\sigma) + u_\sigma, \nonumber \\ 
a_\sigma^{2^i} \tau_i &=& \tau_{i-1} u_\sigma^{2^{i-1}} + \zeta_i \eta_R(u_\sigma), \, \, \, i \geq 1. \nonumber 
\end{eqnarray}
\end{enumerate}
\end{thm}
\begin{proof}
The coproduct formula for $b''$ is the same as the one for $b$: 
$$\psi(b'') = \sum_{i \geq 0} b''^{2^i} \otimes \xi_i.$$
Similar to Theorem~\ref{thm:xizeta}, part (1) can be proved by computing $\psi(\psi(c))$ in two ways: 
$$\psi(c) \otimes 1 + \sum_{i \geq 0} \psi(b'')^{2^i} \otimes \xi_i = \psi(\psi(c)) = c \otimes 1 \otimes 1 + \sum_{i \geq 0} b''^{2^i} \otimes \psi(\tau_i).$$

For part (2), similar to the proof of Theorem~\ref{thm:xizeta}, we consider the map on cohomology 
$$\Hm^\bigstar (B_\mathbb{R} \mathbb{Z}/2) \to \Hm^\bigstar(\mathbb{RP}^\infty) = \Hm_\bigstar[[r]].$$
This map is induced by the map $\mathbb{RP}^\infty \hookrightarrow B_\mathbb{R} \mathbb{Z}/2^{C_2} \to B_\mathbb{R}\mathbb{Z}/2$.  The image of $b'' \in \Hm^\sigma(B_\mathbb{R}\mathbb{Z}/2)$ is the same as the one we found in Theorem~\ref{thm:xizeta}, $r^2u_\sigma + ra_\sigma$.  To find the image of $c \in \Hm^\sigma(B_\mathbb{R}\mathbb{Z}/2)$,  note that the only terms in $\Hm_\bigstar[[r]]$ of degree $-\sigma$ are $\{ru_\sigma, a_\sigma\}$.  Hu and Kriz showed that depending our choices, $c$ can either map to $ru_\sigma$ or $ru_\sigma + a_\sigma$.  Assume that we have chosen $c$ so that $c \mapsto ru_\sigma$ (the relationship between $\tau_i$ and $\zeta_i$ is going to be the same regardless of this choice).  There are two ways to compute $\psi(c)$.  On one hand, 
$$\psi(c) = c \otimes 1 + \sum_{i \geq 0} b''^{2^i} \otimes \tau_i = (ru_\sigma) \otimes 1 + \sum_{i \geq 0} (r^2u_\sigma + ra_\sigma)^{2^i} \otimes \tau_i.$$
On the other hand, 
$$\psi(c) =  \psi(ru_\sigma) = \eta_R(u_\sigma) \psi(r) = \eta_R(u_\sigma) \sum_{i \geq 0} r^{2^i} \otimes \zeta_i.$$
Comparing the coefficients of $r^{2^i}$ for both expressions produces the recursion formulas, as desired. 
\end{proof}

\begin{rem}\rm
The proof above also shows that in the ring $\Hm^\bigstar(B_\mathbb{R} \mathbb{Z}/2)$, there is the relation $c^2 = b'' u_\sigma + c\as$, regardless of the choice of $c$.  
\end{rem}

Using the formulas in Theorem~\ref{thm:xizeta} and Theorem~\ref{thm:tauzeta}, one can show that in both $\mathcal{A}_\bigstar^{cc}$ and $\mathcal{A}_\bigstar^m$, the $\xi_i$ and $\tau_i$ generators are related by the formula 
$$\tau_i^2 = \tau_{i+1}a_\sigma + \xi_{i+1} \eta_R(\us) \, \, \, \, \, \text{(\cite[Proposition~6.37]{HuKriz})}.$$
This is the last ingredient needed to compute the Hopf algebroids $(H_\bigstar^c, \mathcal{A}_\bigstar^{cc})$ and $(H_\bigstar^m, \mathcal{A}_\bigstar^m)$. 

\begin{thm} [Corollary 6.40 and Theorem 6.41 in \cite{HuKriz}] \hfill \label{thm:AccAm} 
\begin{enumerate}
\item The $\mathscr{M}$-Hopf algebroid $(H_\bigstar^c, \mathcal{A}_\bigstar^{cc})$ can be described by  
$$\mathcal{A}_\bigstar^{cc} = H_\bigstar^{c}[\xi_i, \tau_i]/(\tau_0 \as = u_\sigma + \eta_R(u_\sigma), \, \tau_i^2 = \tau_{i+1}\as + \xi_{i+1} \eta_R(\us)),$$
with comultiplications
\begin{enumerate}
\item $\displaystyle \psi(\xi_i) = \sum_{0 \leq j \leq i} \xi_{i-j}^{2^j} \otimes \xi_j$;
\item $\displaystyle \psi(\tau_i) = \tau_i \otimes 1+ \sum_{0 \leq j \leq i} \xi_{i-j}^{2^j} \otimes \tau_j$.
\end{enumerate}

\item The Hopf algebroid $(H^m_\bigstar, \mathcal{A}_\bigstar^m)$ can be described by   
$$\mathcal{A}_\bigstar^m = \Hm_\bigstar [\xi_i, \tau_i]/(\tau_0\as = \us + \eta_R(\us), \tau_i^2 = \tau_{i+1} \as + \xi_{i+1} \eta_R(\us)).$$
The comultiplications are the same as the ones in $(H_\bigstar^c, \mathcal{A}_\bigstar^{cc})$.  The right unit on the elements $\frac{\theta}{u_\sigma^i a_\sigma^j} \in \Hm_\bigstar$ is given by the formula
$$\eta_R\left(\frac{\theta}{u_\sigma^i a_\sigma^j} \right) = \partial\left(\frac{1}{(u_\sigma + \tau_0\as)^{i+1}\as^{j+1}} \right), \, \, \, i \geq 0, \, j \geq 0,$$
where $\partial$ is the boundary map in Remark~\ref{rem:theta}. 
\end{enumerate}
\end{thm}


There are certain extensions involving the Hopf algebroids $(H_\bigstar^c, \mathcal{A}_\bigstar^{cc})$ and $(H_\bigstar^m, \mathcal{A}_\bigstar^m)$ that will produce change of rings theorems.  As we will see later, these change of rings theorems will greatly simplify the computation of the $C_2$-equivariant May and Adams spectral sequences of $BP_\mathbb{R}$.  

Let $P_\bigstar = \mathbb{F}_2[\xi_i \,| \, i \geq 1]$.  Then with the coproduct formula 
$$\psi(\xi_i) = \sum_{0 \leq j \leq i } \xi_{i-j}^{2^j} \otimes \xi_j,$$
$(\mathbb{F}_2, P_\bigstar)$ is a $\mathbb{F}_2$-Hopf algebra.  Similarly, $(\mathbb{F}_2[\as], P_\bigstar[\as])$ is a $\mathbb{F}_2[\as]$-Hopf algebra.  

\begin{prop}[Proposition 6.29 and Theorem 6.41(b) in \cite{HuKriz}] \hfill\label{HopfLambda} 
\begin{enumerate}
\item There is an extension of $\mathscr{M}$-Hopf algebroids
$$(\mathbb{F}_2[\as], P_\bigstar[\as]) \to (H_\bigstar^c, \mathcal{A}_\bigstar^{cc}) \to (H_\bigstar^c, \Lambda_\bigstar^{cc}),$$
where 
$$\Lambda^{cc}_\bigstar = H_\bigstar^c[\tau_i]/(\tau_i^2 = \tau_{i+1}\as),$$
with structure formulas
\begin{enumerate}
\item $\tau_i$ are primitive; 
\item $\eta_R(a_\sigma) = a_\sigma$; 
\item $\eta_R(u_\sigma)= u_\sigma + \tau_0a_\sigma.$
\end{enumerate}

\item There is an extension of Hopf algebroids 
$$(\mathbb{F}_2[\as], P_\bigstar[\as]) \to (H_\bigstar^m, \mathcal{A}_\bigstar^{m}) \to (H_\bigstar^m, \Lambda_\bigstar^{m}),$$
where 
$$\Lambda^m_\bigstar = \Hm_\bigstar[\tau_i]/(\tau_i^2 = \tau_{i+1}a_\sigma).$$
The structure formulas for $\tau_i$, $\as$, and $\us$ are the same as the ones in $(H_\bigstar^c, \Lambda_\bigstar^{cc})$.  
\end{enumerate}
\end{prop}

\subsection{The equivariant Adams spectral sequence}\label{EquivariantASS}
We now introduce the equivariant Adams spectral sequences that are associated to the Hopf algebroids $(\Hm_\bigstar, \mathcal{A}^m_\bigstar)$ and $(H^{c}_\bigstar, \mathcal{A}^{cc}_\bigstar)$, respectively.  

Given a $C_2$-equivariant spectrum $X$, we can resolve $X$ by $\Hm$.  The resulting resolution is the equivariant Adams resolution of $X$.  The spectral sequence associated to this resolution is the \textbf{$C_2$-equivariant Adams spectral sequence} associated to $(\Hm_\bigstar, \mathcal{A}^m_\bigstar)$.  Hu and Kriz observed that the equivariant Steenrod algebra $\mathcal{A}^m_\bigstar$ is a free $\Hm_\bigstar$-module, hence flat over $\Hm_\bigstar$.  From this, they concluded that the $E_2$-page of the $(\Hm_\bigstar, \mathcal{A}^m_\bigstar)$-Adams spectral sequence can be identified as 
$$\Ext_{\mathcal{A}^m_\bigstar}(\Hm_\bigstar, \Hm_\bigstar X) \Longrightarrow (\pi_\bigstar^{C_2} X)_2^{\wedge}. $$
(cf. \cite[Corollary~6.47]{HuKriz}).  Similar to the classical Adams spectral sequence, the equivariant Adams spectral sequence will converge in nice cases.  In particular, it will converge for $X$ a finite $C_2$-spectrum, $MU_\mathbb{R}$, or $BP_\mathbb{R}$. 

On the other hand, by work of Greenlees \cite{GreenleesThesis, Greenlees1988, Greenlees1990}, we can also form the classical Adams resolution of the underlying spectrum of $X$, and then apply the functor $F(E{C_2}_+, -)$ to the classical Adams tower.  The resulting spectral sequence associated to this new tower has $E_2$-page 
$$\Ext_{\mathcal{A}^{cc}_\bigstar} (H^c_\bigstar, H^{cc}_\bigstar X) \Longrightarrow (\pi_\bigstar^{C_2} F(E{C_2}_+, X))_2^\wedge.$$
Again, this equivariant Adams spectral sequence will converge in our cases of interest. 

By applying the functor $F(E{C_2}_+, -)$ to the equivariant Adams resolution of $X$ by $\Hm$, we produce a map of towers, hence a map between the two Adams spectral sequences
$$\begin{tikzcd}
\Ext_{\mathcal{A}^m_\bigstar}(\Hm_\bigstar, \Hm_\bigstar X) \ar[r, Rightarrow] \ar[d] & (\pi_\bigstar^{C_2} X)_2^{\wedge} \ar[d] \\ 
\Ext_{\mathcal{A}^{cc}_\bigstar} (H^c_\bigstar, H^{cc}_\bigstar X) \ar[r, Rightarrow] & (\pi_\bigstar^{C_2} F(E{C_2}_+, X))_2^\wedge.
\end{tikzcd}$$

On the $E_2$-page, the map 
$$\Ext_{\mathcal{A}^m_\bigstar}(\Hm_\bigstar, \Hm_\bigstar X) \to \Ext_{\mathcal{A}^{cc}_\bigstar} (H^c_\bigstar, H^{cc}_\bigstar X)$$
is induced from the map $(\Hm_\bigstar, \mathcal{A}^m_\bigstar) \to (H^c_\bigstar, \mathcal{A}^{cc}_\bigstar)$ of Hopf algebroids.

When $X = BP_\mathbb{R}$, we can simplify the $E_2$-pages of both Adams spectral sequences using Proposition~\ref{prop:BPRxi} and Proposition~\ref{HopfLambda}:
\begin{eqnarray*}
\Ext_{\mathcal{A}^{m}_\bigstar}(\Hm_\bigstar, \Hm_\bigstar BP_\mathbb{R}) &=& \Ext_{\Lambda^m_\bigstar}(\Hm_\bigstar, \Hm_\bigstar), \\ 
\Ext_{\mathcal{A}^{cc}_\bigstar} (H^c_\bigstar, H^{cc}_\bigstar BP_\mathbb{R}) &=& \Ext_{\Lambda^{cc}_\bigstar} (H^c_\bigstar, H^{c}_\bigstar).
\end{eqnarray*}
Here, $\Lambda^m = \Hm \wedge_{BP_\mathbb{R}} \Hm$ and $\Lambda^{cc} = F(E{C_2}_+, \Hm \wedge_{BP_\mathbb{R}} \Hm)$.

\section{The Slice Spectral Sequence of $BP_\mathbb{R}$} \label{sec:SliceSSHFPSSBPR}

We will now discuss the slice spectral sequence and the homotopy fixed point spectral sequence of $BP_\mathbb{R}$.

\subsection{The slice spectral sequence of $BP_\mathbb{R}$}
For definitions and properties of the slice filtration, we refer the readers to \cite[Section~4]{HHRKervaire}.  We will be interested in both the integer-graded and the $RO(C_2)$-graded slice spectral sequence of $BP_\mathbb{R}$:
\begin{eqnarray*}
E_2^{s, t} = \pi_{t -s}^{C_2} P^t_t BP_\mathbb{R} &\Longrightarrow& \pi_{t-s}^{C_2} BP_\mathbb{R} \\
E_{2}^{s, V} = \pi_{V-s}^{C_2} P^{\dim V}_{\dim V} BP_\mathbb{R} &\Longrightarrow& \pi_{V - s}^{C_2} BP_\mathbb{R}
\end{eqnarray*}

The gradings are the Adams grading, with $r^{\text{th}}$-differentials $d_r: E_2^{s, t} \to E_2^{s+r, t+(r-1)}$ and $d_r: E_2^{s, V} \to E_2^{s + r, V+(r-1)}$, respectively. 

To produce the $E_2$-page of the slice spectral sequence, we compute the slice sections $P^t_t BP_\mathbb{R}$.  Let $\bar{v}_i \in \pi_{(2^i-1)\rho_2}BP_\mathbb{R}$ be the equivariant lifts of the usual generators $v_i \in \pi_{2(2^i-1)} BP$.  Using the method of twisted monoid rings \cite[Section~2.4]{HHRKervaire}, we construct the $A_\infty$-map 
$$S^0[\bar{v}_1, \bar{v}_2, \ldots]  \to BP_\mathbb{R}.$$
This map has the property that after taking $\pi_*^u(-)$, it becomes an isomorphism.  Using terminologies developed in \cite{HHRKervaire}, this map is a \textit{multiplicative refinement} of $\pi_*^uBP_\mathbb{R}$.  Furthermore, this multiplicative refinement produces the slice sections of $BP_\mathbb{R}$.  The following result is a special case of the Slice Theorem (\cite[Theorem~6.1]{HHRKervaire}) applied to $BP_\mathbb{R}$. 

\begin{prop}\label{prop:BPRSlices}
The only nonzero slice sections of $BP_\mathbb{R}$ are $P^{2n}_{2n} BP_\mathbb{R}$, where $n \geq 0$.  They are 
$$P_{2n}^{2n} BP_\mathbb{R} = \left( \bigvee_I S^{n\rho_2} \right) \wedge H\underline{\mathbb{Z}}$$
where $I$ is the indexing set consisting of all monomials of the form $\bar{v}_1^{i_1} \bar{v}_2^{i_2}\bar{v}_3^{i_3}\cdots$ with 
$$n\rho_2= (\rho_2)i_1 + (3\rho_2)i_2 + (7\rho_2)i_3 + \cdots.$$
\end{prop}

Proposition~\ref{prop:BPRSlices} shows that computing the $E_2$-page of the slice spectral sequence of $BP_\mathbb{R}$ can be reduced to computing the coefficient group $H\underline{\mathbb{Z}}_\bigstar^{C_2}$.   

\begin{df}[The classes $a_V$ and $u_V$] \rm
Let $V$ be a representation of $G$ with $\dim V = d$.  
\begin{enumerate}
\item $a_V \in \pi_{-V}^G S^0$ is the map corresponding to the inclusion $S^0 \hookrightarrow S^V$ induced by $\{0\} \subset V$.
\item If $V$ is oriented, $u_V \in \pi_{d-V}^G \HZ$ is the class corresponding to the generator of $H_d^{G}(S^V; \HZ)$.  
\end{enumerate}
\end{df}

A comprehensive computation for the coefficient ring of $\HZ_\bigstar^{C_2}$ can be found in \cite{DuggerKR}. 
\begin{thm}[Theorem 2.8 in \cite{DuggerKR}]\label{thm:HZcoeff}
Figure~\ref{fig:HZCoeff} shows the coefficient ring $H\underline{\mathbb{Z}}_{p + q\sigma}^{C_2}$.
\begin{figure}
\begin{center}
\makebox[\textwidth]{\includegraphics[trim={0cm 11cm 7cm 4cm},clip, scale = 0.7]{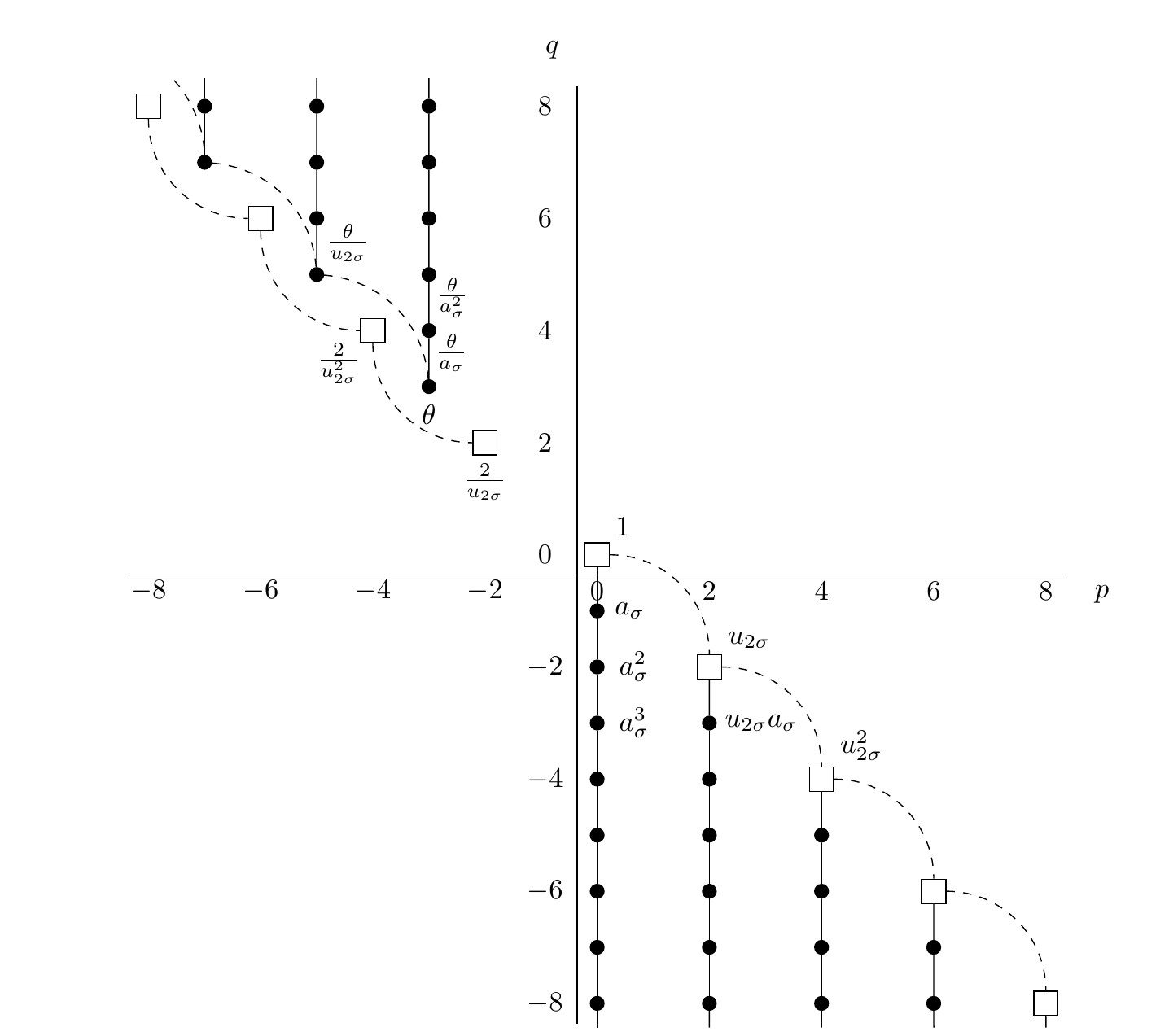}}
\end{center}
\begin{center}
\caption{Coefficient ring of $H\underline{\mathbb{Z}}_{p + q\sigma}^{C_2}$.  Multiplication by $a_\sigma$ are drawn with solid lines and multiplication by $u_{2\sigma}$ are drawn with dashed lines (not all multiplicative structures are drawn).}
\label{fig:HZCoeff}
\end{center}
\end{figure}
Its product structures are as follows: 
\begin{enumerate}
\item In the range $p \geq 0$, $H\mathbb{Z}_{p + q\sigma}^{C_2}$ is the polynomial algebra $\mathbb{Z}[u_{2\sigma}, a_\sigma]/(2a_\sigma)$.  
\item In the range $p < 0$, the class $\alpha = \frac{2}{u_{2\sigma}} \in H_{-2 + 2\sigma}^{C_2}$ is killed by $a_\sigma$ and is infinitely $u_{2\sigma}$ divisible; the class $\theta \in H_{-3 + 3\sigma}^{C_2}$ is killed by $u_{2\sigma}$ and $a_\sigma$ and it is infinitely $u_{2\sigma}$ divisible and $a_\sigma$ divisible.  
\end{enumerate}
\end{thm}

Proposition~\ref{prop:BPRSlices} and Theorem~\ref{thm:HZcoeff} enable us to compute the $E_2$-page of the $RO(C_2)$-graded slice spectral sequence of $BP_\mathbb{R}$.  In particular, the positive part is the polynomial algebra $\mathbb{Z}[\bar{v}_i, u_{2\sigma}, a_\sigma]/(2a_\sigma)$ with 
$$\begin{array}{rll} 
|\bar{v}_i| &=& (0, (2^i -1) + (2^i -1)\sigma),\\ 
|u_{2\sigma}| &=& (0, 2 - 2 \sigma),\\ 
|a_\sigma| &=& (1, 1 - \sigma). 
\end{array}$$

The $E_2$-page of the integer graded slice spectral sequence is the sub-algebra consisting of all the elements that have integer degrees in $t -s$.  It is concentrated in the first quadrant with a vanishing line of slope 1.  
  
\begin{prop}\label{prop:BPRSliceSS}
In the $RO(C_2)$-grade slice spectral sequence for $\pi_\bigstar^{C_2}BP_\mathbb{R}$, $a_\sigma$ and $\bar{v}_i$ are permanent cycles.  The differentials $d_i (u_{2\sigma}^{2^{k-1}})$ are zero for $i < 2^{k+1} -1$, and 
$$d_{2^{k+1} -1} (u_{2\sigma}^{2^{k-1}}) =  \bar{v}_ka_\sigma^{2^{k+1}-1}. $$
\end{prop}
\begin{proof}
This is a special case of Hill--Hopkins--Ravenel's Slice Differential Theorem (\cite[Theorem~9.9]{HHRKervaire}), applied to when $G = C_2$.   
\end{proof}

In Figure~\ref{fig:BPRKRSSS}--\ref{fig:BPRER3SSS}, we draw the first three sets of differentials of the integer-graded slice spectral sequence.  To organize this information in a clean way, we have disassembled the spectral sequence into ``stages'', corresponding to the differentials $d_3$, $d_7$, $d_{15}$, $\ldots$.  At each stage, the important surviving torsion elements are shown.  Many classes with low filtrations (i.e., those on the 0-line) are not drawn because they are not torsion, and hence won't be important for the purpose of this paper. 

\begin{figure}
\begin{center}
\makebox[\textwidth]{\includegraphics[trim={4.75cm 14.5cm 11cm 2cm}, clip, scale = 0.9]{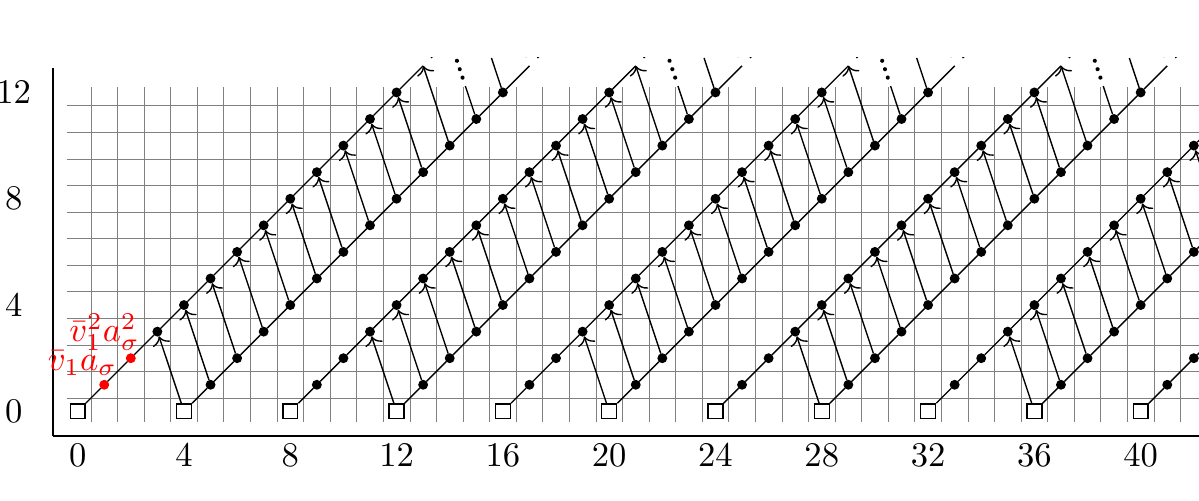}}
\end{center}
\begin{center}
\caption{Important $d_3$-differentials and surviving torsion classes in $\SliceSS(BP_\mathbb{R}).$}
\label{fig:BPRKRSSS}
\end{center}
\end{figure}

\begin{figure}
\begin{center}
\makebox[\textwidth]{\includegraphics[trim={0cm 7cm 3cm 2cm},clip, scale = 0.75]{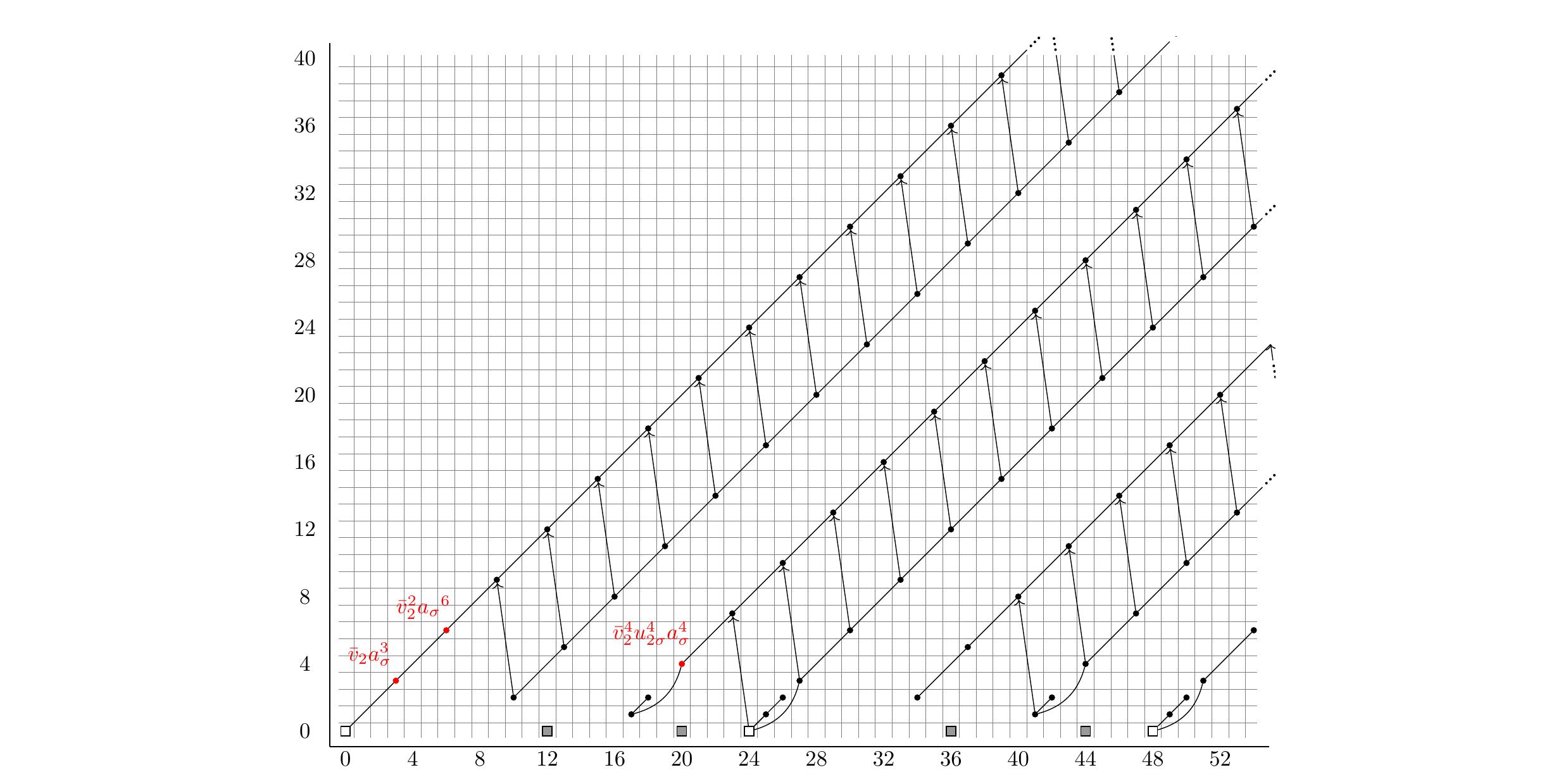}}
\end{center}
\begin{center}
\caption{Important $d_7$-differentials and surviving torsion classes in $\SliceSS(BP_\mathbb{R})$.}
\label{fig:BPRER2SSS}
\end{center}
\end{figure}

\begin{figure}
\begin{center}
\makebox[\textwidth]{\includegraphics[trim={0cm 7cm 3cm 2cm},clip, scale = 0.75]{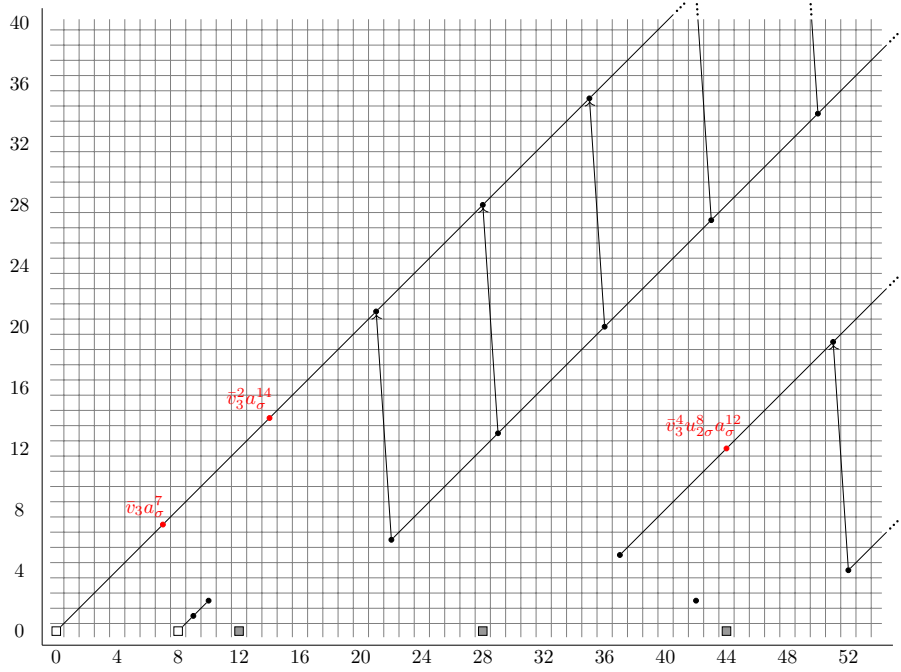}}
\end{center}
\begin{center}
\caption{Important $d_{15}$-differentials and surviving torsion classes in $\SliceSS(BP_\mathbb{R})$.}
\label{fig:BPRER3SSS}
\end{center}
\end{figure}

\subsection{$\SliceSS(BP_\mathbb{R}) \to \HFPSS(BP_\mathbb{R})$}   \label{sec:SliceToHFPSS}

The homotopy fixed point spectral sequence of $BP_\mathbb{R}$ is also going to be useful to us.  It is a spectral sequence that computes the $C_2$-equivariant homotopy groups of $F(E{C_2}_+, BP_\mathbb{R})$.  The integer-graded homotopy fixed point spectral sequence for $BP_\mathbb{R}$ is 
$$E_2^{s, t} = H^s(C_2, \pi_{t}^uBP_\mathbb{R}) \Longrightarrow \pi_{t-s}^{C_2} F(E{C_2}_+, BP_\mathbb{R}).$$
Just like the slice spectral sequence, there is also an $RO(C_2)$-graded version of this, with $E_2$-page
$$E_2^{s, V} = H^s(C_2; \pi_0(S^{-V} \wedge BP_\mathbb{R})) \Longrightarrow \pi_{V-s}^{C_2}F(E{C_2}_+, BP_\mathbb{R}).$$
For a more general discussion of the $RO(G)$-graded homotopy fixed point spectral sequence for any equivariant $G$-spectrum $X$, see \cite[Section 2.3]{HillMeier}.  By \cite[Corollary 4.7]{HillMeier}, the $E_2$-page of the $RO(C_2)$-graded homotopy fixed point spectral sequence of $BP_\mathbb{R}$ is isomorphic to the polynomial algebra 
$$\mathbb{Z}[\bar{v}_i, u_{2\sigma}^\pm, a_\sigma]/(2a_\sigma).$$
The differentials are given by 
\begin{eqnarray}
d_{2^{k+1}-1}(u_{2\sigma}^{2^{k-1}}) &=& \bar{v}_ka_\sigma^{2^{k+1}-1}, \label{eqn:HFPSSDiff}  \\ 
d_{2^{k+1}-1}(u_{2\sigma}^{-2^{k-1}}) &=& d_{2^{k+1}-1}(u_{2\sigma}^{-2^k}\cdot u_{2\sigma}^{2^{k-1}}) \nonumber \\ 
&=& u_{2\sigma}^{-2^k}  d_{2^{k+1}-1}(u_{2\sigma}^{2^{k-1}}) \nonumber \\ 
&=&\bar{v}_ku_{2\sigma}^{-2^k}a_\sigma^{2^{k+1}-1}. \nonumber
\end{eqnarray}
They can be obtained by equivariant primary cohomology operations (see \cite[Lemma 3.34]{HuKriz}).  The readers might have noticed at this point that the differentials on the positive powers of $u_{2\sigma}$ are the same as the differentials in the slice spectral sequence.  Indeed, there is a map $\SliceSS(BP_\mathbb{R}) \to \HFPSS(BP_\mathbb{R})$ that induces an isomorphism in a certain range.  

To explain this map of spectral sequences, we will first construct a map of towers.  Let $X$ be a $C_2$-spectrum.  Let $\mathcal{S}^S_{>n}$ denote the localizing subcategory generated by all the slice cells of dimension $> n$, and $\mathcal{S}^P_{>n}$ the localizing subcategory generated by all the spheres of dimension $> n$.  When $n \geq 0$, $\mathcal{S}^P_{> n} \subseteq \mathcal{S}^S_{>n}$, and this gives a natural map of towers 
$$\text{Post}^\bullet(X) \to P^\bullet(X)$$
from the Postnikov tower of $X$ to the slice tower of $X$.  Non-equivariantly, this map is an isomorphism, because the slice tower is the Postnikov tower when we forget the $C_2$-action.  It follows that after taking $F(E{C_2}_+, -)$ to both towers, the horizontal map in the following diagram is an isomorphism: 
$$\begin{tikzcd}
F(E{C_2}_+, \text{Post}^\bullet(X)) \ar[r,
"\cong"] &F(E{C_2}_+, P^\bullet(X))\\
& P^\bullet(X) \ar[u]. 
\end{tikzcd}$$
The top-left tower, $F(E{C_2}_+, \text{Post}^\bullet(X))$, is the tower for constructing the homotopy fixed point spectral sequence.  It follows that the vertical map induces a map of $RO(C_2)$-graded spectral sequences:
$$\SliceSS(X) \to \text{HFPSS}(X).$$

\begin{prop}\label{SliceHFPSS}
When $X = BP_\mathbb{R}$, the map $\SliceSS(BP_\mathbb{R}) \to \text{HFPSS}(BP_\mathbb{R})$, considered as a map of integer graded spectral sequences, induces an isomorphism on the $E_2$-page on or below the line of slope 1. 
\end{prop}
\begin{proof}

The map of sections
$$P^{2t}_{2t} BP_\mathbb{R} \to F(E{C_2}_+, P^{2t}_{2t} BP_\mathbb{R})$$
is the map $X \to F(E{C_2}_+, X)$ induced by the collapse $E{C_2}_+ \to S^0$.  To prove the desired isomorphism, it suffices to show that the map 
$$\pi_{2t-s}^{C_2}(S^{t\rho_2} \wedge H\underline{\mathbb{Z}}) \to \pi_{2t-s}^{C_2} F(E{C_2}_+,S^{t\rho_2} \wedge H\underline{\mathbb{Z}})$$
is an isomorphism for all $2t-s \geq s$, or $t \geq s$.  This is equivalent to showing that the map
$$\pi_{(t - s) - t\sigma}^{C_2}  H \underline{\mathbb{Z}} \to \pi_{(t - s) - t\sigma}^{C_2}F(E{C_2}_+, H\underline{\mathbb{Z}}) \cong \pi_{(t - s) - t\sigma}^{C_2}F(E{C_2}_+, H\mathbb{Z}) $$
is an isomorphism for all $t \geq s \geq 0$, which is true by Lemma~\ref{lem:coeffHc}.
\end{proof}

\begin{lem}\label{lem:coeffHc}
The coefficient ring $\pi_\bigstar^{C_2}F(E{C_2}_+, H\mathbb{Z})$ is the polynomial algebra 
$$\mathbb{Z}[u_{2_\sigma}^\pm, a_\sigma]/(2a_\sigma). $$
The map 
$$\pi_{p + q\sigma}^{C_2} \HZ \to \pi_{p + q\sigma}^{C_2} F(E{C_2}_+, \HZ) \cong \pi_{p + q\sigma}^{C_2} F(E{C_2}_+, H\mathbb{Z}),$$
is an isomorphism when $p \geq 0$, sending $u_{2\sigma} \mapsto u_{2\sigma}$, $a_\sigma \mapsto a_\sigma$, and zero when $p< 0$.  
\end{lem}
\begin{proof}
This is a standard computation.  We refer the readers to \cite[Corollary 4.7]{HillMeier} and \cite[Appendix B]{DuggerKR} for more details than what is written here.  The key observation is that computing $\pi_\bigstar^{C_2}F(E{C_2}_+, H\mathbb{Z})$ is equivalent to computing $H^*(C_2; \oplus_{r \in \mathbb{Z}} \, \text{sgn}^{\otimes r})$, where $\text{sgn}$ is the integral sign representation.  The elements $u_{2\sigma} \in \pi^{C_2}_{2 - 2\sigma} H\underline{\mathbb{Z}}$ and $a_\sigma \in \pi^{C_2}_{-\sigma} H\underline{\mathbb{Z}}$, under the map $H\underline{\mathbb{Z}} \to F(E{C_2}_+, H\mathbb{Z})$, become the elements $u_{2\sigma} \in H^0(C_2; \text{sgn}^{\otimes 2})$ and $a_\sigma \in H^1(C_2; \text{sgn})$.
\end{proof}

\begin{rem}\rm
In \cite[Section I.9]{UllmanThesis}, Ullman proves a general isomorphism result for any $G$-spectrum $X$.  When $G = C_2$, his isomorphism range is the region \textit{slightly below} the line of slope 1.  In our case, however, $BP_\mathbb{R}$ has pure and isotropic slices and nonnegative slice sections.  When this happens, we can extend his isomorphism range to be \textit{on or below} the line of slope 1.  Since this line is also the vanishing line of $\SliceSS(BP_\mathbb{R})$, the map $\SliceSS(BP_\mathbb{R}) \to \HFPSS(BP_\mathbb{R})$ is an inclusion on the $E_2$-page. \end{rem}

\begin{rem}\rm
Using Lemma~\ref{lem:coeffHc}, one can further show that the map of $RO(C_2)$-graded spectral sequences
$$\begin{tikzcd}
\SliceSS(BP_\mathbb{R}) \ar[d] \ar[r, Rightarrow] & \pi_\bigstar^{C_2} BP_\mathbb{R} \ar[d, "\cong"] \\ 
\HFPSS(BP_\mathbb{R}) \ar[r, Rightarrow] & \pi_\bigstar^{C_2} F(E{C_2}_+, BP_\mathbb{R})
\end{tikzcd}$$
is an inclusion on the $E_2$-page and an isomorphism in this range.  It turns out that in $\HFPSS(BP_\mathbb{R})$, everything outside of this isomorphism range dies by the differentials in ~(\ref{eqn:HFPSSDiff}).  As a result, we obtain an equivalence $BP_\mathbb{R} \simeq F(E{C_2}_+, BP_\mathbb{R})$.  This is called the strong completion theorem (or the homotopy fixed point theorem) of $BP_\mathbb{R}$.  It is proved by Hu and Kriz in \cite[Theorem 4.1]{HuKriz}.  As noted in their paper, the homotopy fixed point theorem holds for $MU_\mathbb{R}$ as well, but not for $BP_\mathbb{R}\langle n \rangle$.  It fails for $BP_\mathbb{R}\langle n \rangle$ because not everything outside the isomorphism range dies in $\HFPSS(BP_\mathbb{R}\langle n \rangle)$.  In particular, $\pi_* BP_\mathbb{R}\langle n \rangle ^{hC_2}$ is not bounded below.   
\end{rem}



\section{The Equivariant May Spectral Sequence} \label{sec:EquivMaySS}

We will now construct the equivariant May spectral sequence of $BP_\mathbb{R}$ by filtering the equivariant dual Steenrod algebra. 

\subsection{The equivariant May spectral sequence with respect to $A^m_{\bigstar}$}\label{sec:GenuineEMay}
Recall from Section~\ref{EquivariantASS} that for a $C_2$-spectrum $X$ with good properties, its equivariant Adams spectral sequence with respect to $(\Hm_\bigstar, A^m_{\bigstar})$ has $E_2$-page
$$E_2 = \Ext_{A^m_{\bigstar}}(\Hm_\bigstar, \Hm_\bigstar X) \Longrightarrow (\pi_\bigstar^{C_2}X)_2^{\wedge}.$$
To compute the Ext-groups on the $E_2$-page, we filter the Hopf algebroid $(\Hm_\bigstar, A^m_{\bigstar})$ by powers of the ideal $(a_\sigma)$ and set the filtration of the element $\xi_i^{2^j}$ to be $2^{j+1} -1$.  We also filter the $(\Hm_\bigstar, A^m_{\bigstar})$-comodule $\Hm_\bigstar X$ to make it compatible with the filtration on $(\Hm_\bigstar, A^m_{\bigstar})$.  

\begin{df}\rm
The spectral sequence 
$$E_1 = \Ext_{\text{gr}_\bullet A^m_{\bigstar}}(\text{gr}_\bullet \Hm_\bigstar, \text{gr}_\bullet \Hm_\bigstar X) \Longrightarrow \Ext_{A^m_{\bigstar}}(\Hm_\bigstar, \Hm_\bigstar X)$$
is the \textit{$C_2$-equivariant May spectral sequence of $X$ with respect to $(\Hm_\bigstar, A^m_{\bigstar})$}.  It is abbreviated by $\EMaySS_{A^m_{\bigstar}}(X)$.    
\end{df}

We are interested in the case when $X = BP_\mathbb{R}$.  In this case, the $E_2$-page of the equivariant Adams spectral sequence simplifies:
$$E_2 = \Ext_{A^m_{\bigstar}}(\Hm_\bigstar, \Hm_\bigstar BP_\mathbb{R}) =  \Ext_{\Lambda^m_\bigstar}(\Hm_\bigstar, \Hm_\bigstar) \Longrightarrow (\pi_\bigstar^{C_2} BP_\mathbb{R})_2^{\wedge}. $$
Here, $\Lambda^m_\bigstar = \pi_\bigstar^{C_2} \left( \Hm \wedge_{BP_\mathbb{R}} \Hm \right)$, and there is a quotient map 
$$A^m_{\bigstar} \to \Lambda^m_\bigstar,$$
which quotients out the $\xi_i$-generators.  The filtration on $A_\bigstar^m$ induces a filtration on $\Lambda_\bigstar^m$, which is also by powers of $(a_\sigma)$.  It follows that the equivariant May spectral sequence for $BP_\mathbb{R}$ has $E_1$-page
$$E_1 = \Ext_{\text{gr}_\bullet \Lambda^m_\bigstar}(\text{gr}_\bullet \Hm_\bigstar, \text{gr}_\bullet \Hm_\bigstar) \Longrightarrow \Ext_{\Lambda^m_\bigstar}(\Hm_\bigstar, \Hm_\bigstar).$$
To compute this $E_1$-page and its differentials, we use Proposition~\ref{HopfLambda}.  Denote the corresponding class of $\tau_i$ in the cobar complex $C_{\Lambda^m_\bigstar}(\Hm_\bigstar)$ by $w_i := [\tau_i]$.  

\begin{prop}\label{prop:genuineMayE2}
The $E_1$-page of $\EMaySS_{\mathcal{A}_\bigstar^m}(BP_\mathbb{R})$ is the polynomial ring $\Hm_\bigstar[w_i]$.  The positive part of the $E_1$-page (the elements in stems $p+ q\sigma$ with $p \geq 0$) is the polynomial ring 
$$\mathbb{Z}/2[u_\sigma, a_\sigma]\left[ w_i\right].$$

The filtration of each element is given in Table~\ref{tab:table1}. 
\begin{table}[h]
\centering
\begin{tabular}{cccc}
\toprule 
& stem & Adams filtration & May filtration \\ 
\midrule 
$w_i = {[}\tau_i {]}$ & $(2^i -1)\rho_{C_2}$ & $1$ & $0$ \\
$a_\sigma$ & $-\sigma$ & $0$ & $1$\\
$u_\sigma$ & $1 - \sigma$ & $0$ & $0$ \\ 
\bottomrule \\
\end{tabular}
\caption{Adams and May filtrations of elements.}
\label{tab:table1}
\end{table}
\end{prop}
\begin{proof}
Immediate from Proposition~\ref{HopfLambda}. 
\end{proof}

\begin{prop}\label{prop:genuineMayDiff}
In $\EMaySS_{\mathcal{A}_\bigstar^m}(BP_\mathbb{R})$, the classes $w_i$ and $a_\sigma$ are permanent cycles.  The class $u_\sigma^{2^n}$ ($n \geq 0$) supports a differential of length $(2^{n+1} -1)$: 
$$d_{2^{n+1} -1}(u_\sigma^{2^n}) = w_na_\sigma^{2^{n+1}-1}.$$
\end{prop}
\begin{proof}
Since the $\tau_i$ generators are primitive and $\eta_R(a_\sigma) = a_\sigma$, the classes $w_i$ and $a_\sigma$ are permanent cycles.  To obtain the differentials on the classes $u_\sigma^{2^n}$, we use the right unit formula for $u_\sigma$.  By Proposition~\ref{HopfLambda}, the right unit formula for $u_\sigma$ is 
$$\eta_R(u_\sigma) = u_\sigma + \tau_0a_\sigma.$$
This translates to the $d_1$-differential 
$$d_1(u_\sigma) = [\tau_0]a_\sigma = w_0 a_\sigma.$$
In general, taking the right unit formula to the $2^n$-th power yields the formula
\begin{eqnarray*}
\eta_R(u_\sigma^{2^n}) &=& \eta_R(u_\sigma)^{2^n} \\ 
&=& u_\sigma^{2^n} + \tau_0^{2^n}a_\sigma^{2^n} \\ 
&=& u_\sigma^{2^n} + \tau_n a_\sigma^{2^{n+1} -1}.
\end{eqnarray*}
where for the last equality we have repeatedly used the relations $\tau_i^2 = \tau_{i+1}a_\sigma$ in $\Lambda^m_\bigstar$.  This produces the $d_{2^{n+1}-1}$-differential on $u_\sigma^{2^n}$, as desired. 
\end{proof}

\subsection{The equivariant May spectral sequence with respect to $\mathcal{A}^{cc}_\bigstar$}

Everything we did in the previous section can be done with respect to $(H^c_\bigstar, \mathcal{A}_\bigstar^{cc})$ as well.  Consider the equivariant Adams spectral sequence 
$$E_2 = \Ext_{\mathcal{A}^{cc}_\bigstar}(H^c_\bigstar, H^c_\bigstar X) \Longrightarrow (\pi_\bigstar^{C_2} X)_2^{\wedge}.$$
We can filter the Hopf algebroid $(H^c_\bigstar, \mathcal{A}^{cc}_\bigstar)$ by powers of the ideal $(a_\sigma)$ to obtain a similar equivariant May spectral sequence. 

\begin{df}\rm
The spectral sequence 
$$E_1 = \Ext_{\text{gr}_\bullet \mathcal{A}_\bigstar^{cc}}(\text{gr}_\bullet H^c_\bigstar, \text{gr}_\bullet H^{c}_\bigstar X) \Longrightarrow \Ext_{\mathcal{A}_\bigstar^{cc}}(H^c_\bigstar, H^c_\bigstar X)$$
is the \textit{$C_2$-equivariant May spectral sequence for $X$ with respect to $(H^c_\bigstar, \mathcal{A}^{cc}_\bigstar)$}.  It is abbreviated by $\EMaySS_{\mathcal{A}_\bigstar^{cc}}(X)$.
\end{df}

When $X = BP_\mathbb{R}$, we can make the same simplifications as we did in Section~\ref{sec:GenuineEMay}.  Let $\Lambda^{cc}_\bigstar := \pi_\bigstar^{C_2}F(E{C_2}_+, \Hm \wedge_{BP_\mathbb{R}} \Hm)$.  The $E_2$-page of the equivariant Adams spectral sequence and $E_1$-page of the equivariant May spectral sequence for $BP_\mathbb{R}$ are equal to 
$$E_2 = \Ext_{\Lambda^{cc}_\bigstar} (H^c_\bigstar, H^c_\bigstar) \Longrightarrow \pi_\bigstar^{C_2} (BP_\mathbb{R})_2^{\wedge}$$
and 
$$E_1 = \Ext_{\text{gr}_\bullet \Lambda^{cc}_\bigstar} (\text{gr}_\bullet H^c_\bigstar, \text{gr}_\bullet H^c_\bigstar) \Longrightarrow \Ext_{\Lambda^{cc}_\bigstar} (H^c_\bigstar, H^c_\bigstar)$$
respectively.  As before, denote the corresponding class of $\tau_i$ in the cobar complex $C_{\Lambda^{cc}_\bigstar}(H^c_\bigstar)$ by $w_i = [\tau_i]$.

\begin{prop}
The $E_1$-page of $\EMaySS_{\mathcal{A}^{cc}_\bigstar} (BP_\mathbb{R})$ is the polynomial ring 
$$H^c_\bigstar[w_i] = \mathbb{Z}/2[u_\sigma^\pm, a_\sigma][w_i], $$
where the filtration of each element is the same as before (see Proposition~\ref{prop:genuineMayE2}).  
\end{prop}
\begin{proof}
The claim follows directly from \cite[Proposition 6.29]{HuKriz}. 
\end{proof}

\begin{prop}\label{prop:BorelMayDiff}
In $\EMaySS_{\mathcal{A}^{cc}_\bigstar}(BP_\mathbb{R})$, the classes $w_i$ and $a_i$ are permanent cycles.  The classes $u_\sigma^{2^n}$ and $u_\sigma^{-2^n}$ ($n \geq 0$) support differentials of length $(2^{n+1} -1)$: 
\begin{eqnarray*}
d_{2^{n+1} -1}(u_\sigma^{2^n}) &=& w_na_\sigma^{2^{n+1}-1}, \\ 
d_{2^{n+1} -1}(u_\sigma^{-2^n}) &=& w_nu_\sigma^{-2^{n+1}}a_\sigma^{2^{n+1}-1}.
\end{eqnarray*}
\end{prop}
\begin{proof}
The proof is exactly the same as the proof of Proposition~\ref{prop:genuineMayDiff}.  The differentials follow from the right unit formulas
$$\begin{array}{rll}
\eta_R(u_\sigma) &=& u_\sigma + \tau_0a_\sigma, \\ 
\eta_R(u_\sigma^{-1}) &=& (u_\sigma + \tau_0a_\sigma)^{-1} = u_\sigma^{-1} + \tau_0u_\sigma^{-2}a_\sigma \,\,\, \text{(modulo higher powers of $a_\sigma$)},
\end{array}$$
and the relation $\tau_i^2 = \tau_{i+1} a_\sigma$ in $\Lambda^{cc}_\bigstar$.   
\end{proof}

\subsection{$\EMaySS_{\mathcal{A}^{cc}_\bigstar}(BP_\mathbb{R})$ and $\HFPSS(BP_\mathbb{R})$}
The equivariant May spectral sequence $\EMaySS_{\mathcal{A}^{cc}_\bigstar}(BP_\mathbb{R})$ and the homotopy fixed point spectral sequence $\HFPSS(BP_\mathbb{R})$ have the same $E_2$-page and differentials under the correspondence $w_i \leftrightarrow \bar{v}_i$.  This is first observed in (7.1) and (7.2) of \cite{HuKriz}.  The only slight difference is that in $\EMaySS_{\mathcal{A}^{cc}_\bigstar}(BP_\mathbb{R})$, instead of a single $\mathbb{Z}$-class, we have a $w_0$-tower of $\mathbb{Z}/2$-classes.  Our goal in this section is to prove this correspondence. 

\begin{thm}\label{thm:AccHFPSS}
The $C_2$-equivariant May spectral sequence for $BP_\mathbb{R}$ with respect to $(H^c_\bigstar, \mathcal{A}^{cc}_\bigstar)$ is isomorphic to the associated-graded homotopy fixed point spectral sequence for $BP_\mathbb{R}$. 
\end{thm}

\begin{rem}\rm
By the associated-graded homotopy fixed point spectral sequence, we mean that whenever we see a $\mathbb{Z}$-class on the $E_2$-page, we replace it by a tower of $\mathbb{Z}/2$-classes.  
\end{rem}

\begin{proof}
Consider the following diagram of spectral sequences:   
$$\begin{tikzcd}[column sep=large]
\text{algebraic }a_\sigma\text{-Bockstein} \arrow[r, leftrightarrow, "\cong", "\text{collapse}"'] &\text{homotopy }a_\sigma\text{-Bockstein} \ar[d, equal]\\
\EMaySS_{\mathcal{A}^{cc}_\bigstar}(BP_\mathbb{R}) \ar[u, equal] \ar[r, leftrightarrow, dashed, " \cong"] & \text{HFPSS}(BP_\mathbb{R}).
\end{tikzcd}$$

We will explain each arrow in the diagram one by one: \\

\noindent (1)  $\EMaySS_{\mathcal{A}^{cc}_\bigstar}(BP_\mathbb{R}) = \text{algebraic } a_\sigma\text{-Bockstein}$.  This is by definition: the equivariant May spectral sequence for $(H^c_\bigstar, \mathcal{A}^{cc}_\bigstar)$ is defined by filtering $\Lambda^{cc}_\bigstar$ by powers of $(a_\sigma)$, which is the algebraic $a_\sigma$-Bockstein. \\

\noindent (2) $\text{Homotopy }a_\sigma\text{-Bockstein} = \text{HFPSS}(BP_\mathbb{R})$.  This is proven in \cite[Lemma 4.8]{HillMeier}.  We include their proof here because it is nice and short.  Start with the cofiber sequence
$$S^0 \vee S^0 \longrightarrow S^0 \stackrel{a_\sigma}{\longrightarrow} S^{\sigma}.$$
Taking $F(-, BP_\mathbb{R})$ yields the new sequence 
$$\Sigma^{-\sigma} BP_\mathbb{R} \stackrel{a_\sigma}{\longrightarrow} BP_\mathbb{R} \longrightarrow BP_\mathbb{R} \vee BP_\mathbb{R}.$$
The homotopy $a_\sigma$-Bockstein is the spectral sequence associated to this cofiber sequence.  The key observation is that in the cofiber sequence 
$$S(n\sigma)_+ \longrightarrow S^0 \stackrel{a_\sigma^n}{\longrightarrow} S^{n\sigma},  $$ 
$S(n\sigma)_+$ is the $(n-1)$-skeleton for the standard equivariant decomposition of $E{C_2}_+$, which is used to construct the homotopy fixed point spectral sequence.  Taking $F(-, BP_\mathbb{R})$ again, we obtain the following commutative diagram:
$$\begin{tikzcd}
\Sigma^{-(n+1)\sigma} BP_\mathbb{R} \ar[d, "a_\sigma"]\ar[r, "a_\sigma^{n+1}"] & BP_\mathbb{R} \ar[r] \ar[d, "\text{id}"] & F(S((n+1)\sigma)_+, BP_\mathbb{R}) \ar[d] \\
\Sigma^{-n\sigma} BP_\mathbb{R} \ar[r, "a_\sigma^n"] & BP_\mathbb{R} \ar[r] & F(S(n\sigma)_+, BP_\mathbb{R}).
\end{tikzcd}$$
It follows that the towers for constructing the homotopy $a_\sigma$-Bockstein spectral sequence and the $\HFPSS(BP_\mathbb{R})$ are the same. \\

\noindent (3) Algebraic $a_\sigma$-Bockstein $\cong$ Homotopy $a_\sigma$-Bockstein.  Consider the following diagram: 
$$\begin{tikzcd}
\vdots & \vdots & \vdots \\ 
\ar[u] (H^c)^{\wedge 2} \wedge \Sigma^{-\sigma} BP_\mathbb{R} \ar[r, "a_\sigma"] & \ar[u] (H^c)^{\wedge 2} \wedge BP_\mathbb{R} \ar[r] &  \ar[u] (H^c)^{\wedge 2} \wedge  (BP_\mathbb{R} \vee BP_\mathbb{R} )\\
\ar[u] H^c \wedge \Sigma^{-\sigma} BP_\mathbb{R} \ar[r, "a_\sigma"] & \ar[u] H^c \wedge BP_\mathbb{R} \ar[r] & \ar[u] H^c \wedge ( BP_\mathbb{R} \vee BP_\mathbb{R} )\\
\ar[u] \Sigma^{-\sigma} BP_\mathbb{R} \ar[r, "a_\sigma"] & \ar[u] BP_\mathbb{R} \ar[r] & \ar[u] BP_\mathbb{R} \vee BP_\mathbb{R}. 
\end{tikzcd}$$
The vertical direction is the Adams resolution by $H^c = F(E{C_2}_+, H\mathbb{F}_2)$, and the horizontal direction is filtering by powers of $a_\sigma$ (the $a_\sigma$-Bockstein).  There are two ways to compute $\pi_\star^{C_2}(BP_\mathbb{R})_2^\wedge$ from $BP_\mathbb{R} \vee BP_\mathbb{R}$: we can either first use the horizontal filtration, then the vertical filtration, or first use the vertical filtration, and then the horizontal filtration.  This produces the following commutative diagram of spectral sequences
$$\begin{tikzcd}[row sep=huge, column sep = huge]
\Ext_{\mathcal{A}^{cc}_\bigstar}(H^c_\bigstar, H^c_\bigstar(BP_\mathbb{R} \vee BP_\mathbb{R})) \ar[rr, Rightarrow, "\text{algebraic }a_\sigma\text{-Bockstein}"] \ar[d, Rightarrow, "\text{Adams (collapse)}"] && \Ext_{\mathcal{A}^{cc}_\bigstar}(H^c_\bigstar, H^c_\bigstar BP_\mathbb{R}) \ar[d, Rightarrow, "\text{Adams (collapse)}"] \\
\ar[rr, Rightarrow, "\text{homotopy }a_\sigma\text{-Bockstein}"] (\pi_\bigstar^{C_2} BP_\mathbb{R} \vee BP_\mathbb{R})_2^\wedge && (\pi_\bigstar^{C_2} F(E{C_2}_+, BP_\mathbb{R}))^\wedge_2=(\pi_\bigstar^{C_2} BP_\mathbb{R})^\wedge_2.
\end{tikzcd}
$$

The vertical spectral sequences are from the Adams (vertical) filtrations, and the horizontal spectral sequences are from the $a_\sigma$-Bockstein (horizontal).  The left vertical spectral sequence collapses by degree reasons.  In particular, the integer-graded part is the non-equivariant Adams spectral sequence computing $\pi_*BP$, which collapses.  The right vertical spectral sequence collapses by degree reasons as well.  For both Adams spectral sequences, \cite[Theorem 4.11]{HuKriz} and our computation show there are no hidden $a_\sigma$-extensions.  The top spectral sequence is the $a_\sigma$-Bockstein associated with the cofiber sequences
$$(H^c)^{\wedge n} \wedge \Sigma^{-\sigma} BP_\mathbb{R} \stackrel{a_\sigma}{\to} (H^c)^{\wedge n} \wedge BP_\mathbb{R} \to (H^c)^{\wedge n} \wedge (BP_\mathbb{R} \vee BP_\mathbb{R})$$
for $n \geq 1$.  When we are computing the Ext groups, it is exactly the same as filtering $(H^c_\star, A^{cc}_\star)$ by powers of $a_\sigma$.  Therefore this spectral sequence is the algebraic $a_\sigma$-Bockstein, or in other words, the $C_2$-equivariant May spectral sequence with respect to $\mathcal{A}^{cc}_\bigstar$.  Finally, the bottom arrow is the homotopy $a_\sigma$-Bockstein, which is the homotopy fixed point spectral sequence by the previous discussion.  The collapse of the two Adams spectral sequences (and no hidden $a_\sigma$-extensions) implies that the algebraic $a_\sigma$-Bockstein is isomorphic to the associated-graded homotopy $a_\sigma$-Bockstein, as desired.  

\end{proof}

\subsection{$\EMaySS_{\mathcal{A}^m_\bigstar}(BP_\mathbb{R})$ and $\SliceSS(BP_\mathbb{R})$}

The map $(\Hm_\bigstar, \mathcal{A}^m_\bigstar) \to (H^c_\bigstar, \mathcal{A}^{cc}_\bigstar)$ induces maps of the corresponding May and Adams spectral sequences:
$$\begin{tikzcd}
\Ext_{\text{gr}_\bullet \Lambda^{m}_\bigstar} (\text{gr}_\bullet \Hm_\bigstar, \text{gr}_\bullet \Hm_\bigstar) \ar[d] \ar[r, Rightarrow, "\text{May}"] &  \Ext_{\Lambda^{m}_\bigstar} (\Hm_\bigstar, \Hm_\bigstar) \ar[r, Rightarrow, "\text{Adams}"] \ar[d] &\pi_\bigstar^{C_2} (BP_\mathbb{R})_2^{\wedge} \ar[d, "\text{id}"] \\
\Ext_{\text{gr}_\bullet \Lambda^{cc}_\bigstar} (\text{gr}_\bullet H^c_\bigstar, \text{gr}_\bullet H^c_\bigstar) \ar[r, Rightarrow, "\text{May}"] &  \Ext_{\Lambda^{cc}_\bigstar} (H^c_\bigstar, H^c_\bigstar) \ar[r, Rightarrow, "\text{Adams}"] &\pi_\bigstar^{C_2} (BP_\mathbb{R})_2^{\wedge}. 
\end{tikzcd}$$
For the purpose of finding Hurewicz images, we restrict our attention to the maps between integer-graded spectral sequences.  They are induced from the map of integer-graded Hopf algebroids $(\Hm_*, A^m_*)\to (H^c_*, A^{cc}_*)$: 
$$\begin{tikzcd}
\Ext_{\text{gr}_\bullet \Lambda^{m}_*} (\text{gr}_\bullet \Hm_*, \text{gr}_\bullet \Hm_*) \ar[d] \ar[r, Rightarrow, "\text{May}"] &  \Ext_{\Lambda^{m}_*} (\Hm_*, \Hm_*) \ar[r, Rightarrow, "\text{Adams}"] \ar[d] &\pi_*^{C_2} (BP_\mathbb{R})_2^{\wedge} \ar[d, "\text{id}"] \\
\Ext_{\text{gr}_\bullet \Lambda^{cc}_*} (\text{gr}_\bullet H^c_*, \text{gr}_\bullet H^c_*) \ar[r, Rightarrow, "\text{May}"] &  \Ext_{\Lambda^{cc}_*} (H^c_*, H^c_*) \ar[r, Rightarrow, "\text{Adams}"] &\pi_*^{C_2} (BP_\mathbb{R})_2^{\wedge}. 
\end{tikzcd}$$
The $E_2$-page of $\EMaySS_{A^m_*}(BP_\mathbb{R})$ is the subring of the polynomial ring $\Hm_\bigstar[w_i]$ that contains only integer-graded elements.  For degree reasons, monomials of the form $u_\sigma^{-k}a_\sigma^{-l}w_{i_1}^{e_1}w_{i_2}^{e_2} \cdots w_{i_j}^{e_j}$ with $k > 0$, $l \geq 0$ do not have integer grading.  It follows that the integer graded elements are all contained in the subring
$$\mathbb{Z}/2[u_\sigma, a_\sigma][w_i] \subset \Hm_\bigstar[w_i].$$

\begin{thm}\label{thm:AmSliceSS}
The integer-graded $C_2$-equivariant May spectral sequence of $BP_\mathbb{R}$ with respect to $(\Hm_*, A^m_*)$ is isomorphic to the associated-graded slice spectral sequence of $BP_\mathbb{R}$.  
\end{thm}

\begin{proof}
Consider the following diagram: 
$$\begin{tikzcd}
\EMaySS_{A^m_*}(BP_\mathbb{R}) \ar[d] \ar[r, leftrightarrow, dashed, "\cong"] & \SliceSS(BP_\mathbb{R}) \ar[d]\\
\EMaySS_{A^{cc}_*}(BP_\mathbb{R}) \ar[r, leftrightarrow, "\cong"] & \HFPSS(BP_\mathbb{R}).
\end{tikzcd}$$
The above discussion, together with Proposition~\ref{prop:genuineMayDiff} and Proposition~\ref{prop:BorelMayDiff}, show that the left vertical map is an inclusion on the $E_2$-page and an isomorphism on or below the line of slope 1.  Proposition~\ref{SliceHFPSS} shows that the right vertical map is also an inclusion on the $E_2$-page and an isomorphism on or below the line of slope 1.  Given the isomorphism already established in Theorem~\ref{thm:AccHFPSS}, it follows that $\EMaySS_{A^m_*}(BP_\mathbb{R})$ is isomorphic to the associated-graded $\SliceSS(BP_\mathbb{R})$, as desired. 
\end{proof}

\begin{rem}\rm
The isomorphism in Theorem~\ref{thm:AccHFPSS} is of $RO(C_2)$-graded spectral sequences.  We have only proven the isomorphism in Theorem~\ref{thm:AmSliceSS} as integer-graded spectral sequences.  This is enough for the purpose of proving Hurewicz images. 
\end{rem}

\section{Map of May Spectral Sequences} \label{sec:MaySSMay}

\subsection{Map of dual Steenrod algebras}
The maps $\mathcal{A}_* \to \mathcal{A}_\bigstar^m \to \Lambda_\bigstar^m$ induce maps of Adams $E_2$-pages:
$$E_2(\ASS(\mathbb{S})) \to E_2(\EASS(\mathbb{S})) \to E_2(\EASS(BP_\mathbb{R})). $$
Filtering $\mathcal{A}_*$, $\mathcal{A}_\bigstar^m$, and $\Lambda_\bigstar^m$ compatibly with respect to the map above produces maps of May spectral sequences 
$$\text{MMaySS}(\mathbb{S}) \to \EMaySS(\mathbb{S}) \to \EMaySS(BP_\mathbb{R}).$$
Here, $\text{MMaySS}(\mathbb{S})$ is the modified May spectral sequence, which will be defined in section~\ref{sec:ChangeofFiltration}.  These maps of May spectral sequences will later help us prove our detection theorem for $BP_\mathbb{R}^{C_2}$.


Recall that
\begin{eqnarray*}
\mathcal{A}_* &=& \mathbb{F}_2[\zeta_1, \zeta_2, \ldots], \\ 
\mathcal{A}_\bigstar^m &=& \Hm_\bigstar [\xi_i, \tau_i] / (\tau_i^2 = \tau_{i+1}a_\sigma + \xi_{i+1} \eta_R(u_\sigma) ), \\ 
\Lambda_\bigstar^m &=& \Hm_\bigstar[\tau_i]/(\tau_i^2 = \tau_{i+1}a_\sigma).
\end{eqnarray*}

The following theorem will be used later for
\begin{enumerate}
\item Constructing and computing the map $\text{MMaySS}(\mathbb{S}) \to \EMaySS(BP_\mathbb{R})$ of May spectral sequences. 
\item Computing the images of elements in the classical Adams spectral sequence $\ASS(\mathbb{S})$ under the map 
$$\ASS(\mathbb{S}) \rightarrow \EASS(BP_\mathbb{R}).$$
\end{enumerate}

\begin{thm}\label{thm:SteenrodAlgebraMap}
The composite map $\mathcal{A}_* \to \mathcal{A}^m_\bigstar \to \Lambda^m_\bigstar$ sends the element
$$ \zeta_i^{2^j} \mapsto \tau_{i+j-1} \us^{2^{i+j-1} - 2^j} \as^{2^j -1} \text{ (modulo higher powers of $a_\sigma$)}.$$
\end{thm}
\begin{proof}
By an abuse of notation, we will denote $\zeta_i$ to be the image of $\zeta_i \in \mathcal{A}_*$ in $\mathcal{A}_\bigstar^m$ and $\Lambda_\bigstar^m$.  The following relations hold in $\mathcal{A}^m_\bigstar$ (cf. \cite[Theorem 6.18, Theorem 6.41]{HuKriz} and Theorem~\ref{thm:xizeta}):
\begin{eqnarray}
\as^{2^i} \xi_i &=& \zeta_{i-1}^2 \eta_R(\us) + \zeta_i \as + \xi_{i-1}\us^{2^{i-1}}, \,\,\,\,\,\,  \label{eq:5.1}\\
a_\sigma \tau_0 &=& u_\sigma + \eta_R(u_\sigma).  \,\,\,\,\,\, \label{eq:5.3}
\end{eqnarray}
To prove our claim, we will show using induction on $i$ that 
$$\zeta_i \mapsto \tau_{i-1} \eta_R(\us)^{2^{i-1} -1} \pmod{\xi_1, \xi_2, \ldots}.$$
For the base case when $i = 1$, equations (\ref{eq:5.1}) and (\ref{eq:5.3}) imply
$$\begin{array}{llrll}
&& \as^2\xi_1 &=& \eta_R(\us) + \zeta_1\as + \us \\
&\Longrightarrow& \zeta_1 \as &=& \us + \eta_R(\us) = a_\sigma \tau_0 \pmod{\xi_1, \xi_2, \ldots} \\ 
&\Longrightarrow& \zeta_1 &=& \tau_0 \pmod{\xi_1, \xi_2, \ldots}. 
\end{array}$$
Therefore the base case holds.  Now, suppose we have shown that 
$$\zeta_{i-1} = \tau_{i-2}\eta_R(\us)^{2^{i-2}-1} \pmod{\xi_1, \xi_2, \ldots}.$$
To prove the relation for $\zeta_i$, we use relation (\ref{eq:5.1}) again:
$$\begin{array}{llrll}
&&\zeta_i \as &=& \zeta_{i-1}^2 \eta_R(\us) \pmod{\xi_1, \xi_2, \ldots}  \\ 
&&&=& (\tau_{i-2}\eta_R(\us)^{2^{i-2}-1} )^2 \eta_R(\us) \pmod{\xi_1, \xi_2, \ldots} \, \, \, \, \, \text{(induction hypothesis)}\\
&&&=& \tau_{i-2}^2 \eta_R(\us)^{2^{i-1} -1} \pmod{\xi_1, \xi_2, \ldots}  \\ 
&&&=& (\tau_{i-1} \as) \eta_R(\us)^{2^{i-1} -1} \pmod{\xi_1, \xi_2, \ldots} \\ 
&\Longrightarrow& \zeta_i &=& \tau_{i-1} \eta_R(\us)^{2^{i-1} -1} \pmod{\xi_1, \xi_2, \ldots}, 
\end{array}$$
as desired. 

To finish the proof of the theorem, we need to simplify the expression 
$$\zeta_i^{2^j} = \left(\tau_{i-1} \eta_R(\us)^{2^{i-1} -1} \right)^{2^j} $$
modulo higher powers of $a_\sigma$.  Since $\us + \eta_R(\us) = \tau_0 a_\sigma$, 
$$\us \equiv \eta_R(\us) \, \, \, \, \, \text{(modulo higher powers of $\as$)}. $$
After applying the relation $\tau_{n}^2 = \tau_{n+1} a_\sigma$ $j$-times, we obtain the equality
$$\zeta_i^{2^j} \equiv \tau_{i+j-1} \us^{2^{i+j-1} - 2^j} \as^{2^j -1} \, \, \, \, \, \text{(modulo higher powers of $\as$)},$$ 
as desired. 
\end{proof}

\subsection{Change of filtration}\label{sec:ChangeofFiltration}
The maps $\mathcal{A}_* \to A^m_{\bigstar} \to \Lambda^m_\bigstar$ induce maps of Ext groups 
$$\begin{tikzcd}
\Ext_{\mathcal{A}_*}(\F, \F) \ar[r] \ar[rd] &\Ext_{A^m_{\bigstar}} (\Hm_\bigstar, \Hm_\bigstar)  \ar[d] \\
&\Ext_{\Lambda^m_\bigstar}(\Hm_\bigstar, \Hm_\bigstar).
\end{tikzcd}$$
To analyze these maps, we will construct maps of May spectral sequences:
$$\text{MMaySS}(\mathbb{S}) \to \EMaySS(\mathbb{S}) \to \EMaySS(BP_\mathbb{R}).$$  
We do so by filtering the classical dual Steenrod algebra $\mathcal{A}_*$ and the equivariant dual Steenrod algebras $A^m_{\bigstar}$ and $\Lambda_\bigstar^m$ compatibly with respect to the maps $\mathcal{A}_* \to A^m_{\bigstar} \to \Lambda^m_\bigstar$.  

Recall that in constructing the classical May spectral sequence, the dual Steenrod algebra $\mathcal{A}_* = \F[\zeta_1, \zeta_2, \ldots]$ is filtered by powers of its unit coideal, hence producing an increasing filtration.  More specifically, we can define a grading on $\mathcal{A}_*$ by setting the degree of $h_{i, j} := \zeta_i^{2^j}$ to be 
$$|h_{i, j}| = 2i -1$$ 
and extend additively to all unique representatives.  The increasing filtration associated to this grading is
$$F_0\mathcal{A}_* \subset F_1 \mathcal{A}_* \subset \cdots \subset F_p \mathcal{A}_* \subset F_{p+1} \mathcal{A}_* \subset \cdots \subset \mathcal{A}_*,$$
where at stage $p$, $F_p\mathcal{A}_*$ consists of all elements of total degrees $\leq p$.  


However, in constructing the equivariant May spectral sequence $\EMaySS(\mathbb{S})$ and $\EMaySS(BP_\mathbb{R})$, we filtered $\mathcal{A}^m_\bigstar$ and $\Lambda^m_\bigstar$ by powers of $(a_\sigma)$ and produced decreasing filtrations instead.  To rectify this mismatch of filtrations, we need to change the filtration of the classical dual Steenrod algebra $\mathcal{A}_*$ to make it compatible with the decreasing filtrations on $\mathcal{A}^m_\bigstar$ and $\Lambda^m_\bigstar$.  In particular, it must be a decreasing filtration.  To do this, notice that by Theorem~\ref{thm:SteenrodAlgebraMap}, the element $h_{i, j}$ is sent to 
\begin{eqnarray} \label{hijImage}
h_{i, j} &\mapsto& \tau_{i+j-1} \us^{2^{i+j-1}-2^j}\as^{2^j-1} \, \, \, (\text{Modulo higher powers of } \as).
\end{eqnarray}
We can define a new grading on $\mathcal{A}_*$ on by setting the degree of the generators $h_{i, j}$ to be 
$$|h_{i, j}| = 2^j -1 $$
and extend linearly to all unique representatives.  The decreasing filtration associated to this is 
$$\mathcal{A}_*= F_0\mathcal{A}_* \supset F_1\mathcal{A}_* \supset F_2\mathcal{A}_* \supset \cdots, $$
where $F_p\mathcal{A}_*$ contains elements whose total degrees are $\geq p$.  From (\ref{hijImage}), it is immediate that this filtration is compatible with the decreasing filtration of $A^m_{\bigstar}$ and $\Lambda^m_\bigstar$ with respect to the maps $\mathcal{A}_* \to A^m_{\bigstar} \to \Lambda^m_\bigstar$.  Therefore, we obtain maps of May spectral sequences 
$$\begin{tikzcd}
\text{MMaySS}(\mathbb{S}) \ar[r] \ar[rd] & \EMaySS(\mathbb{S}) \ar[d] \\ 
& \EMaySS(BP_\mathbb{R}).
\end{tikzcd}$$

To compute the $E_1$-page of our modified May spectral sequence $\text{MMaySS}(\mathbb{S})$, consider the coproduct formula for $h_{i, j}$: 
\begin{eqnarray*}
\Psi(h_{i, j}) &=& 1 \otimes h_{i, j} + h_{i, j}\otimes 1+ \sum_{k =1}^{i-1} h_{i-k, j+k} \otimes h_{k, j} \\ 
&=& \underbrace{1 \otimes h_{i, j}}_{\deg =  2^j -1} + \underbrace{h_{i, j} \otimes 1}_{\deg = 2^j -1}+ \underbrace{h_{i-1, j+1} \otimes h_{1, j}}_{\deg = (2^{j+1} -1) + (2^j -1)} + \underbrace{h_{i-2, j+2} \otimes h_{2, j}}_{\deg = (2^{j+2} -1) + (2^j -1)} + \\
&&  \cdots + \underbrace{h_{1, j+i-1} \otimes h_{i-1, j}}_{\deg = (2^{j+i -1} -1) + (2^j -1)}.
\end{eqnarray*}
With the old filtration, $|h_{i,j}| = 2i -1$, and everything in the summation sign on the right is of degree exactly $2i-2$.  After changing to our new filtration, everything in this sum is of degree at least 
\begin{eqnarray*}
|h_{i-1, j+1} \otimes h_{1, j}| &=& (2^{j+1} -1)  + (2^j - 1) \\ 
&>& 2^j -1 \\ 
&=& |h_{i, j}|.
\end{eqnarray*}
Therefore after projecting to the associated-graded $\mathcal{A}_* \to \text{gr}_\bullet \mathcal{A}_*$, the elements $h_{i, j}$ are primitive.  It follows that the $E_1$-page of our modified May spectral sequence is still the polynomial algebra generated by the $h_{i,j}$: 
$$E_1 = \Ext_{\text{gr}_\bullet \mathcal{A}_*} (\F, \F) = \F[\{h_{i, j}\}_{i \geq 1, j \geq 0}] \Longrightarrow \Ext_{\mathcal{A}_*}(\F, \F).$$
The differentials obtained from the coproduct formula for $h_{i, j}$ is now of length $2^{j+1} - 1$:
$$d_{2^{j+1}-1}(h_{i, j}) = h_{i-1, j+1} \otimes h_{1, j}.$$
(before changing the filtration, it was $d_1(h_{i, j}) = \sum_{k =1}^{i-1} h_{i-k, j+k} \otimes h_{k, j}$).  Intuitively, with our new filtration, the differentials are being ``stretched out''.  

\section{Detection Theorem} \label{sec:DetectionTheorem}

\subsection{Detection Theorems for $BP_\mathbb{R}$ and $MU_\mathbb{R}$}

We will now prove our detection theorems for the Hopf, Kervaire, and $\bar{\kappa}$-family by analyzing the map of spectral sequences 
$$\text{MMaySS}(\mathbb{S}) \to \EMaySS_{\mathcal{A}^m_\bigstar}(BP_\mathbb{R}) \cong \SliceSS(BP_\mathbb{R}).$$  
Using Theorem~\ref{thm:SteenrodAlgebraMap}, the following proposition is immediate. 

\begin{prop}\label{prop:mayImages}
On the $E_2$-page of the map $\text{MMaySS}(\mathbb{S}) \to \EMaySS(BP_\mathbb{R}) \cong \SliceSS(BP_\mathbb{R})$, 
\begin{eqnarray*}
h_{1n} &\mapsto& \bar{v}_n \as^{2^n-1} ,\\ 
h_{1n}^2 &\mapsto& \bar{v}_n^2 \as^{2(2^n-1)} ,\\ 
h_{2n}^4 &\mapsto& \bar{v}_{n+1}^4 \us^{2^{n+2}} \as^{4(2^n-1)}.
\end{eqnarray*}
\end{prop}

\begin{prop}\label{prop:SliceSSsurvive}
In the slice spectral sequence of $BP_\mathbb{R}$, the classes $\bar{v}_n \as^{2^n-1}$, $\bar{v}_n^2 \as^{2(2^n-1)}$, and $\bar{v}_{n+1}^4 \us^{2^{n+2}} \as^{4(2^n-1)}$ survive to the $E_\infty$-page.  
\end{prop}
\begin{proof}
As discussed in Proposition~\ref{prop:BPRSliceSS}, all of the differentials in the slice spectral sequence for $BP_\mathbb{R}$ are completely classified by the slice differential theorem (\cite[Theorem 9.9]{HHRKervaire}).  They are 
$$d_{2^{k+1}-1}(\us^{2^k}) = \as^{2^{k+1}-1}\bar{v}_k, \, \, \, \, \,  k \geq 1.$$
The longest possible differentials that could possibly kill the classes mentioned are differentials of length $2^{n+1}-1$.  The survival of these classes is a straightforward computation (see Figure~\ref{fig:ER2SSS}). 
\begin{figure}
\begin{center}
\makebox[\textwidth]{\includegraphics[trim={0cm 7cm 3cm 2cm},clip, scale = 0.75]{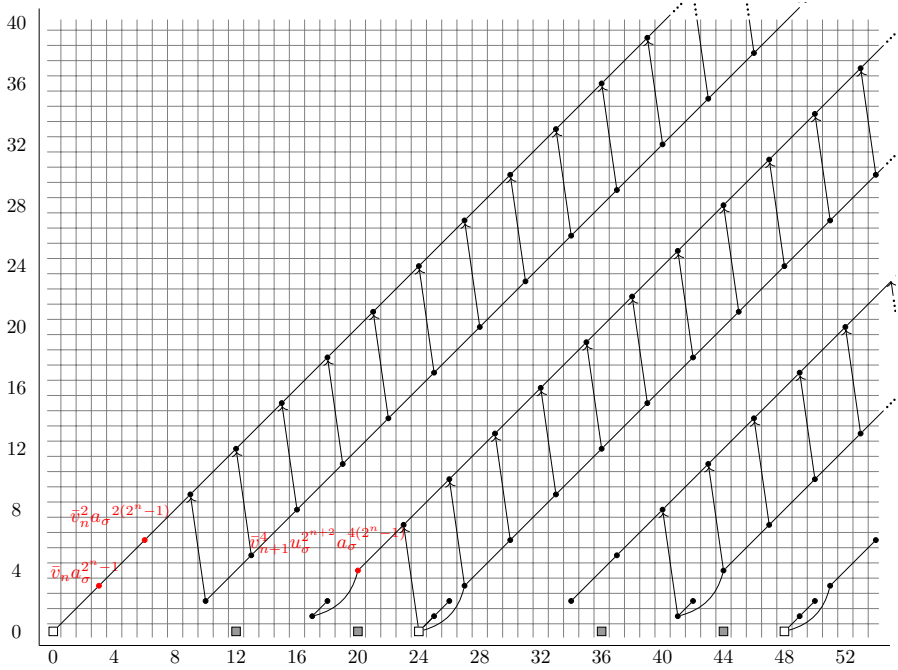}}
\end{center}
\begin{center}
\caption{A part of the slice spectral sequence for $BP_\mathbb{R}$.  The classes $\bar{v}_n \as^{2^n-1}$, $\bar{v}_n^2 \as^{2(2^n-1)}$, and $\bar{v}_{n+1}^4 \us^{2^{n+2}} \as^{4(2^n-1)}$ survive to the $E_{2^{n+1}}$-page, hence to the $E_\infty$-page.}
\label{fig:ER2SSS}
\end{center}
\end{figure}
\end{proof}

\begin{thm}[Detection of Hopf elements and Kervaire classes]\label{thm:HopfThetan}
If the element $h_{n}$ or $h_{n}^2$ in $\Ext_{\mathcal{A}_*} (\F, \F)$ survives to the $E_\infty$-page of the Adams spectral sequence, then its image under the Hurewicz map $\pi_*\mathbb{S} \to \pi_* BP_\mathbb{R}^{C_2}$ is nonzero.  
\end{thm}
\begin{proof}
In the modified May spectral sequence, the classes $h_{1n}$ and $h_{1n}^2$ are permanent cycles.  Furthermore, since $h_{1n}$ is of Adams filtration 1 and $h_{1n}^2$ is of Adams filtration 2, if they are targets of differentials in the modified May spectral sequence, then the source of these differentials must be of Adams filtrations 0 and 1, respectively.  This is clearly impossible.  Therefore they are not targets of differentials and hence survive to the $E_\infty$-page of the modified May spectral sequence. 

On the $E_\infty$-page of the modified May spectral sequence, these classes represent the unique elements $h_{n} \in \Ext_{\mathcal{A}_*}^{1, 2^n}(\F, \F)$ and $h_{n}^2 \in \Ext_{\mathcal{A}_*}^{2, 2^{n+1}}(\F, \F)$ in their respective bidegrees.  The theorem now follows from Proposition~\ref{prop:mayImages} and Propostion~\ref{prop:SliceSSsurvive}.  
\end{proof}

\begin{thm}[Detection of $\bar{\kappa}$-family]\label{thm:Kappabar}
If the element $g_n \in \Ext_{\mathcal{A}_*}^{4, 2^{n+2}+2^{n+3}}(\F, \F)$ survives to the $E_\infty$-page of the Adams spectral sequence, then its image under the Hurewicz map $\pi_*\mathbb{S} \to \pi_* BP_\mathbb{R}^{C_2}$ is nonzero. 
\end{thm}

\begin{rem}\rm
As mentioned previously, the element $g_1$ survives to the element $\bar{\kappa} \in \pi_{20}\mathbb{S}$.  The element $g_2$ also survives to an element in $\pi_{44}\mathbb{S}$.  They are both detected in $\pi_* BP_\mathbb{R}^{C_2}$.  The fate of the $g_n$ elements for $n \geq 3$ is unknown.
\end{rem}

The proof of Theorem~\ref{thm:Kappabar} requires the following facts: 

\begin{lem}\label{lem:gn} 
The element $g_n$ is the only nonzero element in $\Ext_{\mathcal{A}_*}^{4, 2^{n+2}+2^{n+3}}(\F, \F)$.  
\end{lem}
\begin{proof}
We appeal to the classification theorem of Lin (\cite[Theorem 1.3]{LinExt}).  For indecomposable elements in $\Ext^{4, t}_{\mathcal{A}_*}(\F, \F)$, binary expansion of $t$ shows that $g_n$ is the only indecomposable element in $\Ext^{4, 2^{n+2}+ 2^{n+3}}_{\mathcal{A}_*}(\F, \F)$.  The only possible decomposable elements in the same bidegree as $g_n$ are $h_nc_n$ and $h_{n+1}^2h_{n+2}^2$.  However, they are both zero due to relations in $\Ext^{*,*}_{\mathcal{A}_*}(\F, \F)$.  
\end{proof}

\begin{lem}\label{lem:hijIneq}
Let $1 \leq m \leq n$, if 
$$\deg(h_{i_1, j_1} h_{i_2, j_2} \cdots h_{i_m, j_m}) < \deg (h_{i_1', j_1'} h_{i_2', j_2'} \cdots h_{i_n', j_n'}),$$
then for any $k \geq 1$, 
$$\deg(h_{i_1, j_1+k} h_{i_2, j_2+k} \cdots h_{i_m, j_m+k}) < \deg (h_{i_1', j_1'+k} h_{i_2', j_2'+k} \cdots h_{i_n', j_n'+k}).$$
\end{lem}
\begin{proof}
The given condition translates to the inequality 
$$\sum (2^{j_m} -1) < \sum (2^{j_n'} -1),$$
which, after rearranging, is 
$$ \sum 2^{j_m} < \left(\sum 2^{j_n'}\right) - (n-m).$$
Multiplying both sides of the inequality by $2^k$ produces the new inequality 
\begin{eqnarray}
\sum 2^{j_m + k} &<& \left(\sum 2^{j_n' + k} \right) - 2^k(n-m) \nonumber \\
&<& \left(\sum 2^{j_n'+k} \right) - (n-m). \nonumber
\end{eqnarray}
Rearranging this inequality gives
$$\sum (2^{j_m + k} -1) < \sum (2^{j_n' + k} -1), $$
as desired. 
\end{proof}

\begin{lem}\label{lem:h2n1}
The element $h_{21}^4$ is a permanent cycle in the modified May spectral sequence of the sphere.  
\end{lem}
\begin{proof}
It is clear from the modified May filtration that all the differentials are of odd length.  First, we have 
$$d_3(h_{21}) = h_{11}h_{12},$$
so $d_3(h_{21}^2) = 0$.  In fact, $d_5(h_{21}^2) = 0$ as well.  To show this, notice that $h_{21}^2$ is in tridegree $(s, t, m) = (2, 12, 2)$, where $(s, t)$ is the degree associated to $\Ext_{\mathcal{A}_*}^{s, t}(\F, \F)$ and $m$ is the modified May filtration.  If $h_{21}^2$ supports a $d_5$, then the target must be of tridegree $(3, 12, 7)$.  We will characterize all $h_{i_1j_1}h_{i_2j_2}h_{i_3j_3}$ of this tridegree.  The equations that need to be satisfied are 
\begin{eqnarray*}
(2^{i_1} -1)2^{j_1} + (2^{i_2}-1)2^{j_2} + (2^{i_3} -1)2^{j_3} &=& 12 , \\ 
(2^{j_1} -1) + (2^{j_2}-1) + (2^{j_3} -1) &=& 7.
\end{eqnarray*}
Since $7 = 7 + 0 + 0 = 3+3+1$, it's not hard to check that the only possibility for $h_{i_1j_1}h_{i_2j_2}h_{i_3j_3}$ is $h_{10}h_{20}h_{13}$.  However, this element cannot be the target of a differential because it supports a nontrivial $d_1$-differential 
$$d_1(h_{10}h_{20}h_{13}) = h_{10}^2h_{11}h_{13}.$$
Therefore $d_5(h_{21}^2) = 0$.  

By computations in the cobar complex $C(\mathcal{A}_*)$ (see \cite[Lemma 3.2.10(b)]{RavenelGreen}), we deduce that 
$$d_7(h_{21}^2) = h_{12}^3 + h_{11}^2 h_{13}.$$
(Note that the computation in the cobar complex also shows that $h_{21}^2$ is a $d_5$-cycle.)  By the Leibneiz rule, this differential implies that $d_7(h_{21}^4) = 0$.  

The element $h_{21}^4$ is in tridegree $(s, t, m) = (4, 24, 4)$.  The target of a differential $d_r$ with source $h_{21}^4$ must be of tridegree $(5, 24, 4+r)$.  In particular, it must be a linear combination of elements of the form $h_{i_1j_1} \cdots h_{i_5j_5}$ satisfying the equations
\begin{eqnarray} 
\sum_{k = 1}^5 (2^{i_k} -1)2^{j_k} = 24, \label{eq:5hijRelation2} \\ 
\sum_{k =1}^5 (2^{j_k} -1) = 4 + r. \label{eq:5hijRelation}
\end{eqnarray}
Since $d_7(h_{21}^4) = 0$, $r$ must be at least 9.  Moreover, subtracting Equation~(\ref{eq:5hijRelation}) from Equation~(\ref{eq:5hijRelation2}) yields the equation 
$$\sum_{k = 1}^5 (2^{i_k + j_k} - 2^{j_k+1} + 1) = 20 -r.$$
The left hand side is at least 5 because $i_k \geq 1$.  It follows that $r \leq 15$.  We will now rule out each possibility in the range $9 \leq r \leq 15$ case-by-case: \\

\noindent \textbf{Case 1:} $r = 9$.  Equation~(\ref{eq:5hijRelation}) becomes 
$$\sum_{k =1}^5 (2^{j_k} -1) = 13.$$
The possibilities are
\begin{eqnarray}
13 &=& 7 + 3 + 3 + 0 + 0  \nonumber \\ 
&=& 7+ 3 + 1 + 1 + 1 \nonumber \\ 
&=& 3 + 3 + 3 + 3 + 1. \nonumber
\end{eqnarray}
The second and third possibilities give nothing.  The first possibility gives $h_{10}h_{30}h_{12}^2h_{13}$, which supports a nontrivial $d_1$ differential 
$$d_1(h_{10}h_{30}h_{12}^2h_{13}) = h_{10}^2h_{21}h_{12}^2h_{13}.$$
Therefore $d_9(h_{21}^4) = 0$.  \\

\noindent \textbf{Case 2:} $r = 11$.  Equation~(\ref{eq:5hijRelation2}) becomes 
$$\sum_{k =1}^5 (2^{j_k} -1) = 15.$$
The possibilities are 
\begin{eqnarray}
15 &=& 15 + 0 + 0 + 0  + 0 \nonumber \\
&=& 7 + 7 + 1 + 0 + 0 \nonumber \\ 
&=& 7 + 3 + 3 + 1 + 1 \nonumber \\ 
&=& 3 + 3 + 3 + 3 + 3.\nonumber 
\end{eqnarray}
The first possibility gives $h_{10}^2h_{20}^2h_{14}$.  The second possibility gives $h_{10}^2h_{21}h_{13}^2$ and $h_{11}h_{20}^2h_{13}^2$.  The third possibility gives $h_{11}h_{21}h_{12}^2h_{13}$.  The fourth possibility gives nothing.  Now, we will rule them out one by one.  
\begin{itemize}
\item For $h_{10}^2h_{20}^2h_{14}$, we can first argue using the cobar complex that 
$$d_3(h_{20}^2) = h_{10}^2h_{12} + h_{11}^3.$$
So there is a nontrivial $d_3$-differential 
\begin{eqnarray*}
d_3(h_{10}^2h_{20}^2h_{14})  &=& h_{10}^2h_{14}(h_{10}^2h_{12} + h_{11}^3) \\ 
 &=& h_{10}^4h_{12}h_{14} + h_{10}^2h_{11}^3h_{14} \\ 
 &=& h_{10}^4h_{12}h_{14}. 
\end{eqnarray*}
(the element $h_{10}^2h_{11}^3h_{14} = 0$ on the $E_3$-page because $d_1(h_{20}) = h_{10}h_{11}$).
\item The element $h_{10}^2h_{21}h_{13}^2$ is the target of the $d_1$-differential
$$d_1(h_{10}h_{30}h_{13}^2) = h_{10}^2 h_{21}h_{13}^2,$$
and hence 0 on the $E_3$-page.
\item The element $h_{11}h_{20}^2h_{13}^2$ supports a nontrivial $d_3$-differential
$$d_3(h_{11}h_{20}^2h_{13}^2) = h_{11}h_{13}^2(h_{10}^2h_{12} + h_{11}^3) = h_{11}^4h_{13}^2$$
(the first term in the sum is 0 because $h_{10}h_{11} = 0$ on the $E_3$-page). 
\item The element $h_{11}h_{21}h_{12}^2h_{13}$ supports a nontrivial $d_3$ differential 
$$d_3(h_{11}h_{21}h_{12}^2h_{13}) = h_{11}h_{12}^2h_{13}\, d_3(h_{21})= h_{11}^2h_{12}^3h_{13},$$
hence does not survive past the $E_3$-page. 
\end{itemize}
Therefore $d_{11}(h_{21}^4) = 0$. \\

\noindent \textbf{Case 3:} $r = 13$.  Equation~(\ref{eq:5hijRelation}) becomes 
$$\sum_{k =1}^5 (2^{j_k} -1) = 17.$$
The possibilities are 
\begin{eqnarray*}
17 &=& 15 + 1 + 1 + 0 + 0 \\ 
&=& 7 + 7 + 3 + 0 + 0 \\ 
&=& 7 + 7 + 1 + 1 + 1 \\ 
&=& 7 + 3 + 3 + 3 + 1.
\end{eqnarray*}
The first possibility gives $h_{10}h_{20}h_{11}^2h_{14}$.  The second possibility gives $h_{10}h_{20}h_{12}h_{13}^2$.  The third and fourth possibilities give nothing.  Both elements support nontrivial $d_1$ differentials: 
\begin{eqnarray*}
d_1(h_{10}h_{20}h_{11}^2h_{14}) &=& h_{10}^2h_{11}^3h_{14},\\ 
d_1(h_{10}h_{20}h_{12}h_{13}^2) &=& h_{10}^2h_{11}h_{12}h_{13}^2.
\end{eqnarray*}
Therefore $d_{13}(h_{21}^4) = 0$.  \\

\noindent \textbf{Case 4:} $r = 15$.  Equation~(\ref{eq:5hijRelation}) becomes 
$$\sum_{k =1}^5 (2^{j_k} -1) = 19.$$
The possibilities are 
\begin{eqnarray*}
19 &=& 15 + 3 + 1 + 0 + 0 \\ 
&=& 15 + 1 + 1 + 1 + 1\\ 
&=& 7 + 7 + 3 + 1 + 1 \\ 
&=& 7+3+3+3+3.
\end{eqnarray*}
The first possibility gives $h_{10}^2h_{11}h_{12}h_{14}$.  The second possibility gives $h_{11}^4h_{14}$.  The third possibility gives $h_{11}^2h_{12}h_{13}^2$.  The fourth possibility gives $h_{12}^4h_{13}$.  They do not survive because of the following differentials: 
\begin{eqnarray*}
d_1(h_{10}h_{20}h_{12}h_{14}) &=& h_{10}^2h_{11}h_{12}h_{14}, \\ 
d_3(h_{11}h_{20}^2h_{14}) &=& h_{11}h_{14} \, d_3(h_{20}^2) \\ 
&=& h_{11}h_{14}(h_{10}^2h_{12} + h_{11}^3)\\ 
&=& h_{11}^4h_{14},\\ 
d_3 (h_{11}h_{21}h_{13}^2) &=& h_{11}^2h_{12}h_{13}^2, \\
d_7(h_{12}^3h_{22}) &=& h_{12}^4h_{13}. 
\end{eqnarray*}
Therefore $d_{15}(h_{21}^4) = 0$.  This concludes the proof of the Lemma.  
\end{proof}

\begin{prop}\label{prop:h2n4}
For $n \geq 1$, the elements $h_{2n}^4$ survive to the $E_\infty$-page of the modified May spectral sequence of the sphere. 
\end{prop}
\begin{proof}
We will first show that for all $n \geq 1$, $h_{2n}^4$ is not the target of a differential.  By Proposition~\ref{prop:mayImages} and Proposition~\ref{prop:SliceSSsurvive}, the image of $h_{2n}^4$ on the $E_2$-page under the map $\text{MMaySS}(\mathbb{S}) \to \SliceSS(BP_\mathbb{R})$ is $\bar{v}_{n+1}^4 \us^{2^{n+2}} \as^{4(2^n-1)}$, which survives to the $E_\infty$-page of the slice spectral sequence.  However, if $h_{2n}^4$ is the target of a differential $d_r$, then its image must also be the target of a differential $d_{r'}$, with $r' \leq r$.  This is a contradiction.  Therefore $h_{2n}^4$ is not the target of a differential.  

We now show that $h_{2n}^4$ is also a permanent cycle.  By Lemma~\ref{lem:h2n1}, $h_{21}^4$ is a permanent cycle.  This means that we can find an element $x$ in the cobar complex $C(\mathcal{A}_*)$ of the form
$$x = \underbrace{\xi_2^{2^1} | \xi_2^{2^1} | \xi_2^{2^1} | \xi_2^{2^1}}_{\deg= 4 \cdot (2^1 -1) = 4} + \underbrace{S}_{\deg > 4} $$
such that $d(x) = 0$.  In the expression for $x$, $S$ is a sum containing elements of the form 
$$s = h_{i_1,j_1} \cdots | \cdots | \cdots | \cdots h_{i_k,j_k},$$
with $k \geq 4$ and 
\begin{eqnarray} \label{eqn:h21}
\deg(\xi_2^{2^1}|\xi_2^{2^1}|\xi_2^{2^1}|\xi_2^{2^1}) &=& (2^1 -1) + (2^1 - 1) + (2^1 - 1) + (2^1 -1) \nonumber \\ 
&<& \sum_{k} (2^{j_k}-1) \nonumber \\ 
&=& \deg (s).
\end{eqnarray}
 To show that $h_{2n}^4$ is a permanent cycle, we apply the $\Sq^0$-operation introduced by Nakamura \cite{Nakamura} $(n-1)$-times to the element $x$, obtaining a new element
$$(\Sq^0)^{n-1}(x) = \underbrace{\xi_2^{2^n}|\xi_2^{2^n}|\xi_2^{2^n}|\xi_2^{2^n}}_{\deg = 4 \cdot (2^n -1)} + \underbrace{(\Sq^0)^{n-1}(S)}_{\deg > 4\cdot(2^n-1)}$$
in the cobar complex.  Everything in $(\Sq^0)^{n-1}(S)$ is of the form 
$$(\Sq^0)^{n-1} (s) =  h_{i_1,j_1+n-1} \cdots | \cdots | \cdots | \cdots h_{i_k,j_k + n-1}.$$
Lemma~\ref{lem:hijIneq}, applied to inequality~(\ref{eqn:h21}), shows that  
$$\deg((\Sq^0)^{n-1}(\xi_2^{2^1}|\xi_2^{2^1}|\xi_2^{2^1}|\xi_2^{2^1})) = 4 \cdot (2^n -1) < \deg((\Sq^0)^{n-1}(s)).$$
By \cite[Lemma 3.1]{Nakamura}, the $\Sq^0$-operation preserves the coboundary operator of the cobar complex $C(\mathcal{A}_*)$.  Therefore
$$d((\Sq^0)^{n-1}(x)) = (\Sq^0)^{n-1}d(x) = (\Sq^0)^{n-1}(0) = 0.$$
It follows that $h_{2n}^4$ is a permanent cycle, as desired. 
\end{proof}

\noindent \textit{Proof of Theorem~\ref{thm:Kappabar}.}  By Proposition~\ref{prop:h2n4}, the elements $\{h_{2n}^4 \, | \, n \geq 1\}$ survive to the $E_\infty$-page of the modified May spectral sequence, hence they detect some nonzero elements in $\Ext^{4, 2^{n+2} + 2^{n+3}}_{\mathcal{A}_*}(\F, \F)$.  By Lemma~\ref{lem:gn}, these elements must be $\{g_n \, | \, n \geq 1\}$.  The theorem now follows from Proposition~\ref{prop:mayImages} and Propostion~\ref{prop:SliceSSsurvive}. 
\text{ }\hfill $\square$

\begin{rem}\rm
By \cite[Theorem 4.1]{HuKriz}, the map $\pi_* BP_\mathbb{R}^{C_2} \to \pi_* BP_\mathbb{R}^{hC_2}$ is an isomorphism.  Therefore Theorem~\ref{thm:HopfThetan} and Theorem~\ref{thm:Kappabar} hold for the homotopy fixed point of $BP_\mathbb{R}$ as well. 
\end{rem}

Theorem~\ref{thm:HopfThetan} and Theorem~\ref{thm:Kappabar} combine to produce our detection theorems for $MU_\mathbb{R}$ and $BP_\mathbb{R}$:

\begin{thm}[Detection Theorems for $MU_\mathbb{R}$ and $BP_\mathbb{R}$] \label{thm:MURBPRDetection}
The Hopf elements, the Kervaire classes, and the $\bar{\kappa}$-family are detected by the Hurewicz maps $\pi_*\mathbb{S} \to \pi_* MU_\mathbb{R}^{C_2}$ and $\pi_*\mathbb{S} \to \pi_*^{C_2} BP_\mathbb{R}^{C_2}$.  
\end{thm}

\begin{thm} \label{normDetection}
Let $E$ be an $E_\infty$ H-spectrum.  If the $H$-fixed point spectrum of $E$ detects a class $x \in \pi_*\mathbb{S}$, then the $G$-fixed point spectrum of $(N_H^G E)$ detects $x$ as well. 
\end{thm}

\begin{proof}
This follows from the following commutative diagram: 
$$\begin{tikzcd}
(N_H^G E)^G \ar[r] & (i_H^*N_H^G E)^H = (E \wedge \cdots \wedge E)^H \ar[r] & E^H \\ 
& \mathbb{S} \ar[ul] \ar[u] \ar[ur]&
\end{tikzcd}$$
The first horizontal map is the map from the $G$-fixed point to the $H$-fixed point.  The second horizontal map is obtained by the multiplicative structure on $E$.  Taking $\pi_*(-)$ to the entire diagram gives the maps 
$$\pi_*\mathbb{S} \to \pi_* (N_H^G E)^G \to \pi_* (E^H).$$
Since $x$ maps to a nonzero element in $\pi_*(E^H)$ under the composition map, $x$ must map to a nonzero element in $\pi_*(N_H^G E)^G$ as well.  
\end{proof}

Letting $E = MU_\mathbb{R}$ in Theorem~\ref{normDetection} gives the following:

\begin{cor}\label{cor:DetectionMUG}
For any finite group $G$ containing $C_2$, the $G$-fixed point of $MU^{((G))}$ detects the Hopf elements, the Kervaire classes, and the $\bar{\kappa}$-family. 
\end{cor}

\begin{rem}\rm
Theorem~\ref{normDetection} produces the detection tower
$$\begin{tikzcd}
&& \vdots \ar[d]& \\
&&\pi_*(MU^{((C_{2^n}))})^{C_{2^n}} \ar[d] \\
&& \vdots \ar[d] \\ 
\pi_*\mathbb{S} \ar[rr] \ar[rrd] \ar[rrdd] \ar[rruu]&&\pi_*(MU^{((C_8))})^{C_{8}} \ar[d] \\ 
&&\pi_*(MU^{((C_4))})^{C_{4}} \ar[d] \\ 
&&\pi_*(MU_\mathbb{R})^{C_{2}}.
\end{tikzcd}$$
As we go up the tower, the size of the cyclic group increases, and $\pi_*(MU^{((C_{2^n}))})^{C_{2^n}}$ will detect more classes in the homotopy groups of spheres. 
\end{rem}

\subsection{Detection Theorem for $E\mathbb{R}(n)$}
Recall that the Real Johnson--Wilson theory $E\mathbb{R}(n)$ is constructed from $BP_\mathbb{R}$ by killing $\bar{v}_i$ for $i \geq n+1$ and inverting $\bar{v}_n$.  Its refinement is
$$\begin{tikzcd}
S^0[\bar{v}_1, \bar{v}_2, \ldots] \ar[d] \ar[r] & BP_\mathbb{R} \ar[d] \\
S^0[\bar{v}_1, \ldots, \bar{v}_{n-1}, \bar{v}_n^\pm] \ar[r] & E\mathbb{R}(n).
\end{tikzcd}$$
To prove the detection theorem for $E\mathbb{R}(n)$, we analyze the composite map 
$$\text{MMaySS}(\mathbb{S}) \to \SliceSS(BP_\mathbb{R}) \to \SliceSS(E\mathbb{R}(n)).$$ 

\begin{lem}\label{lem:ER(n)SliceSS}
In the slice spectral sequence for $\pi_\bigstar^{C_2}E\mathbb{R}(n)$, the classes $a_\sigma$, $\bar{v}_1$, $\ldots$, $\bar{v}_{n-1}$, $\bar{v}_n^\pm$ are permanent cycles.  For $1 \leq k \leq n$, the differentials $d_i (u_{2\sigma}^{2^{k-1}})$ are zero for $i < 2^{k+1} -1$, and 
$$d_{2^{k+1} -1} (u_{2\sigma}^{2^{k-1}}) =  \bar{v}_ka_\sigma^{2^{k+1}-1}.$$
The class $u_{2\sigma}^{2^{n}}$ is a permanent cycle.  
\end{lem}
 
\begin{proof}
This is immediate by comparing to the slice spectral sequence of $BP_\mathbb{R}$ (Proposition~\ref{prop:BPRSliceSS}).  For the class $u_{2\sigma}^{2^{n}}$, it is supposed to support a differential to $\bar{v}_{n+1}a_\sigma^{2^{n+2}-1}$.  However, $\bar{v}_{n+1}$ is 0 in the slice spectral sequence for $E\mathbb{R}(n)$.  This implies that $u_{2\sigma}^{2^{n}}$ is a $d_{2^{n+2}-1}$-cycle.  Furthermore, for degree reasons, there are no classes in the appropriate degrees that can be hit by longer differentials from $u_{2\sigma}^{2^{n}}$.  It follows that the class $u_{2\sigma}^{2^{n}}$ is a permanent cycle.  
\end{proof}
 
\begin{thm}[Detection Theorem for $E\mathbb{R}(n)$] \hfill\label{thm:ER(n)Detection}
\begin{enumerate}
\item For $1 \leq k \leq n$, if the element $h_{k}$ or $h_{k}^2$ in $\Ext_{\mathcal{A}_*} (\F, \F)$ survives to the $E_\infty$-page of the Adams spectral sequence, then its image under the Hurewicz map $\pi_*\mathbb{S} \to \pi_* E\mathbb{R}(n)^{C_2}$ is nonzero.  
\item For $1 \leq k \leq n-1$, if the element $g_k \in \Ext_{\mathcal{A}_*}^{4, 2^{n+2}+2^{n+3}}(\F, \F)$ survives to the $E_\infty$-page of the Adams spectral sequence, then its image under the Hurewicz map $\pi_*\mathbb{S} \to \pi_* E\mathbb{R}(n)^{C_2}$ is nonzero. 
\end{enumerate}
\end{thm}
\begin{proof}
By Proposition~\ref{prop:mayImages}, Theorem~\ref{thm:HopfThetan}, and Theorem~\ref{thm:Kappabar}, it suffices to show that the classes $\bar{v}_ka_\sigma^{2^k-1}$ ($1 \leq k \leq n$), $\bar{v}_k^2a_\sigma^{2(2^k-1)}$ ($1 \leq k \leq n$), and $\bar{v}_{k+1}^4u_{\sigma}^{2^{k+2}}a_\sigma^{4(2^k-1)}$ ($1 \leq k \leq n-1$) survive to the $E_\infty$-page of the slice spectral sequence of $E\mathbb{R}(n)$.  This is immediate from Lemma~\ref{lem:ER(n)SliceSS}.
\end{proof}

\begin{rem}\rm
By the homotopy fixed point theorem (\cite[Theorem 10.8]{HHRKervaire}), the spectrum $E\mathbb{R}(n)$ is cofree.  This means that $\pi_* E\mathbb{R}(n)^{C_2} \to \pi_* E\mathbb{R}(n)^{hC_2}$ is an isomorphism.  Therefore Theorem~\ref{thm:ER(n)Detection} holds for the homotopy fixed point of $E\mathbb{R}(n)$ as well. 
\end{rem}

\begin{rem}\rm
The detection theorem for $E\mathbb{R}(n)$ also holds for $v_{n}^{-1} MU_\mathbb{R}$ as it splits as a wedge of suspensions of $E\mathbb{R}(n)$ (with $E\mathbb{R}(n)$ itself being one of the wedge summands). 
\end{rem}

\bibliographystyle{alpha}

\bibliography{references}

\end{document}